\documentclass[12pt]{amsart}

\usepackage{amssymb,latexsym,amscd}
\usepackage{amsthm,amsmath,mathrsfs,amsfonts}
\usepackage[colorlinks,pageanchor]{hyperref}
\newcommand{\cV}{\mathcal{V}}

\newcommand{\cU}{\mathcal{U}}
\newcommand{\cO}{\mathcal{O}}
\newcommand{\cE}{\mathcal{E}}
\newcommand{\cG}{\mathcal{G}}

\newcommand{\cM}{\mathcal{M}}

\newcommand{\cL}{\mathcal{L}}
\newcommand{\cJ}{\mathcal{J}}
\newcommand{\cC}{\mathcal{C}}
\newcommand{\cI}{\mathcal{I}}
\newcommand{\cQ}{\mathcal{Q}}
\newcommand{\cW}{\mathcal{W}}

\newcommand{\sP}{\mathscr{P}}

\newcommand{\Sym}{\mathrm{Sym}}
\newcommand{\Pic}{\mathrm{Pic}}

\newcommand{\Mor}{\mathrm{Mor}}

\newcommand{\Spec}{\mathrm{Spec}}
\newcommand{\Id}{\mathrm{Id}}
\newcommand{\Loc}{\mathrm{Loc}}
\newcommand{\HM}{\mathrm{HM}}
\newcommand{\Nilp}{\mathrm{Nilp}}
\newcommand{\QQuot}{\mathcal{Q}\mathrm{uot}}

\newcommand{\lra}{\longrightarrow}

\newcommand{\ra}{\rightarrow}

\newcommand{\Op}{\mathfrak{Op}}

\newcommand{\RR}{\mathbb{R}}

\newcommand{\End}{\mathrm{End}}

\newcommand{\Hom}{\mathrm{Hom}}

\newcommand{\GL}{\mathrm{GL}}
\newcommand{\PGL}{\mathrm{PGL}}
\newcommand{\SL}{\mathrm{SL}}

\newcommand{\rk}{\mathrm{rk}}

\newcommand{\mydot}{{\scriptstyle{\bullet}}}
\def\map#1{\ \smash{\mathop{\longrightarrow}\limits^{#1}}\ }

\let\into=\hookrightarrow
\let\onto=\twoheadrightarrow
\newcommand{\crg}{r(r-1)(r-2)(g-1)}
\let\tensor=\otimes

\theoremstyle{plain}

\newtheorem{thm}[equation]{Theorem}

\newtheorem{lem}[equation]{Lemma}

\newtheorem{prop}[equation]{Proposition}

\newtheorem{cor}[equation]{Corollary}

\theoremstyle{definition}
\newtheorem{rem}[equation]{Remark}
\newtheorem{ques}[equation]{Question}
\newtheorem{defi}[equation]{Definition}

\numberwithin{equation}{subsection}

\begin{document}

\title[]{Hitchin-Mochizuki morphism, Opers and Frobenius-destabilized vector bundles over curves}

\begin{abstract}
Let $X$ be a smooth projective curve of genus $g \geq 2$ defined
over an algebraically closed field $k$ of characteristic $p >0$. For
$p>\crg$  we
construct an atlas for the locus of all Frobenius-destabilized
bundles of rank $r$ (i.e. we construct all Frobenius-destabilized bundles of
rank $r$ and degree zero up to isomorphism). This is done by exhibiting a
surjective morphism from a certain Quot-scheme onto the locus of
stable Frobenius-destabilized bundles. Further we show that there is
a bijective correspondence between the set of stable vector bundles
$E$ over $X$ such that the pull-back $F^*(E)$ under the Frobenius
morphism of $X$ has maximal Harder-Narasimhan polygon and the set of
opers having zero $p$-curvature. We also show that, after fixing the
determinant, these sets are finite, which enables us to derive the
dimension of certain Quot-schemes and certain loci of stable
Frobenius-destabilized vector bundles over $X$. The finiteness is
proved by studying the properties of the Hitchin-Mochizuki morphism; an alternative approach 
to finiteness has been realized in \cite{chen}.
In particular we prove a generalization of a result of Mochizuki to
higher ranks.
\end{abstract}

\author{Kirti Joshi}

\author{Christian Pauly}

\address{Mathematics Department \\ University of Arizona \\ 617 N Santa Rita \\ Tucson 85721-0089 \\ USA}

\email{kirti@math.arizona.edu}

\address{Laboratoire de Math\'ematiques J.A. Dieudonn\'e \\
Universit\'e de Nice Sophia-Antipolis \\
06108 Nice Cedex 02 \\ France }
\email{pauly@unice.fr}




\maketitle

\bigskip

\begin{flushright}
Dedicated to Vikram B. Mehta
\end{flushright}

\bigskip

\section{Introduction}

\subsection{The statement of the results}
Let $X$ be a smooth projective curve of genus $g \geq 2$ defined
over an algebraically closed field $k$ of characteristic $p >0$. One
of the interesting features of vector bundles in positive
characteristic is the existence of semistable vector bundles $E$
over $X$ such that their pull-back $F^*(E)$ under the absolute
Frobenius morphism $F : X \rightarrow X$ is no longer semistable.
This phenomenon also occurs over base varieties of arbitrary
dimension and is partly responsible for the many difficulties
arising in the construction and the study of moduli spaces of
principal $G$-bundles in positive characteristic. We refer to the
recent survey \cite{langer} for an account of the developments in
this field.

Let us consider the coarse moduli space $\cM(r)$ of $S$-equivalence
classes of semistable vector bundles of rank $r$ and degree $0$ over
a curve $X$ and denote by $\cJ(r)$ the closed subvariety  of
$\cM(r)$ parameterizing semistable bundles $E$ such that $F^*(E)$ is
not semistable. For arbitrary $r,g,p$, besides their non-emptiness
(see \cite{LasP2}), not much is known about the loci $\cJ(r)$. For
example, their dimension and their irreducibility are only known in
special cases or for small values of $r,p$ and $g$; see e.g.
\cite{Ducrohet}, \cite{joshi00},  \cite{JRXY}, \cite{joshi08},
\cite{lange-pauly08}, \cite{LasP2},\cite{Mo}, \cite{Mo2},
\cite{O,O2}, \cite{sun06}. Following \cite{JRXY} one associates to a
stable bundle $E \in \cJ(r)$ the Harder-Narasimhan polygon of the
bundle $F^*(E)$ and defines in that way \cite{shatz} a natural
stratification on the stable\footnote{We note that the
Harder-Narasimhan polygon of $F^*(E)$ may vary when $E$ varies in an
$S$-equivalence class.} locus $\cJ^s(r) \subset \cJ(r)$. Thus the
fundamental question which arises is: what is the geometry of the
locus $\cJ(r)$ and the stratification it carries. Before proceeding
further we would like to recall some notions. A well-known theorem
of Cartier's  (see section 2.1.3) says that there is a one-to-one
correspondence between vector bundles $E$ over $X$ and local systems
$(V, \nabla)$ having $p$-curvature $\psi(V,\nabla)$ zero, which is
given by the two mappings
$$ E \mapsto (F^*(E), \nabla^{can}) \qquad \text{and} \qquad (V, \nabla) \mapsto
V^\nabla = E. $$ Here $V^\nabla$ denotes the sheaf of
$\nabla$-invariant sections and $\nabla^{can}$ the canonical
connection. An important class of local systems (in characteristic
zero) was studied by Beilinson and Drinfeld in their fundamental
work on the geometric Langlands program \cite{BD1}. These local
systems are called opers and they play a fundamental role in the
geometric Langlands program.

An \emph{oper} is a triple $(V,\nabla, V_\mydot)$ (see Definition
\ref{defioper}) consisting of a vector bundle $V$ over $X$, a
connection $\nabla$ on $V$ and a flag $V_\mydot$ satisfying some
conditions. In their original paper \cite{BD1} (see also \cite{BD2})
the authors define opers (with complete flags) over the complex
numbers and identify them with certain differential operators
between line bundles. We note that over a smooth projective curve
$X$ the underlying vector bundle $V$ of an oper as defined in
\cite{BD1} is constant up to tensor product by a line bundle: in the
rank two case the bundle $V$ is also called the {\em Gunning bundle}
$\cG$ (see \cite{Gu}, \cite{Mo}) and is given by the unique
non-split extension of $\theta^{-1}$ by $\theta$ for a
theta-characteristic $\theta$ of the curve $X$. For higher rank $r$
the bundle $V$ equals the symmetric power $\Sym^{r-1}(\cG)$ up to
tensor product by a line bundle. In particular, the bundle $V$ is
non-semistable and we shall denote by $\sP_r^{oper}$ its
Harder-Narasimhan polygon. Opers of rank two appeared in
characteristic $p>0$ in the work of S.~Mochizuki (see \cite{Mo}),
where they appeared as indigenous bundles.  An oper is
\emph{nilpotent} if the underlying connection has nilpotent $p$-curvature (of
exponent $\leq$ rank of the oper). An oper is \emph{dormant} if the
underlying connection has $p$-curvature zero (this terminology is
due to S.~Mochizuki). By definition any dormant oper is nilpotent.
Let $\Op_{\PGL(r)}$ denote the moduli of $\PGL(r)$-opers, i.e. rank-$r$ opers
$(V, \nabla, V_\mydot)$  up to tensor product by rank-$1$ local systems (see section
3.2 for the precise definition) and 
let $\mathrm{Nilp}_r(X) \subset \Op_{\PGL(r)}$ be the subscheme of nilpotent $\PGL(r)$-opers on $X$
(see section 6.1 for the precise definition).
In \cite{Mo}, Mochizuki proved  a foundational result: the scheme of
nilpotent, indigenous bundles is finite (i.e $\mathrm{Nilp}_2(X)$ is finite). 
This result lies at the
center of Mochizuki's $p$-adic uniformization program.  Opers of
higher rank in positive characteristic $p>0$ also appeared in
\cite{JRXY}. We will take a slightly more general definition of
opers by allowing non-complete flags. With our definition the triple
$(F^*(F_*(Q)), \nabla^{can}, V_\mydot)$ associated to any vector
bundle $Q$ over $X$, as introduced in \cite{JRXY} (we note that in
\cite{JRXY} this was shown under assumption that $F_*(Q)$ is stable
if $Q$ is stable; this restriction was removed in
\cite{lange-pauly08} for $Q$ of rank one, and more recently in all
cases by \cite{sun06} see Theorem~\ref{exampleoper}), is an oper,
even a dormant oper.

Our first result is the higher rank case of the finiteness result of
\cite{Mo} --- that the locus of nilpotent $\PGL(r)$-opers is finite (see
Theorem~\ref{nilpfinite}):
\begin{thm}\label{theorem2}
The scheme $\mathrm{Nilp}_r(X)$ is finite.
\end{thm}
We note that Mochizuki allows curves with log structures, but we do
not. Theorem \ref{theorem2} is proved by proving that
$\mathrm{Nilp}_r(X)$ is both affine and proper. One gets affineness
by constructing $\mathrm{Nilp}_r(X)$ as the fiber over $0$ of what
we call the \emph{Hitchin-Mochizuki morphism} (see section 3.3)
$$ \mathrm{HM}: \Op_{\PGL(r)} \lra \bigoplus_{i=2}^r H^0(X, (\Omega_X^1)^{\otimes i}) ,
\qquad (V, \nabla, V_\mydot) \mapsto \left[ \mathrm{Char} \ \psi(V,
\nabla)  \right]^{\frac{1}{p}},$$ which associates to an oper $(V,
\nabla, V_\mydot)$ the $p$-th root of the coefficients of the
characteristic polynomial of the $p$-curvature. That the scheme $\Op_{\PGL(r)}$ of
$\PGL(r)$-opers is affine is due to
Beilinson-Drinfeld (see \cite{BD1}).

This morphism is a generalization of the morphism, first introduced
and studied by S. Mochizuki (see \cite[page 1025]{Mo}) in the rank
two case and for families of curves with logarithmic structures, and
called the \emph{Verschiebung}. We consider a generalization of this
to arbitrary rank and we call it the Hitchin-Mochizuki morphism. Our
approach to Theorem \ref{theorem2} is different from that of
\cite{Mo,Mo2}. The key point is:  Mochizuki's Verschiebung is
really a Hitchin map with affine source and target. From this view point
finiteness is equivalent to properness, and we note that Hitchin
maps (all of them) share a key property: properness. Thus one sees
from this that finiteness of indigenous bundles is a rather natural
consequence of being a fibre of a Hitchin(-Mochizuki) map. Thus it
remains to prove, exactly as is the case with usual Hitchin morphism
(\cite{F}, \cite{N}), that the Hitchin-Mochizuki morphism
$\mathrm{HM}$ is proper (Remark \ref{HMproper}).  We mention that
both sides of the morphism $\mathrm{HM}$ are affine and of the same
dimension.  At any rate, from this point of view we reduce the proof
of finiteness to proving a suitable valuative criterion. This is
accomplished by showing that the underlying local system of an oper
is stable (Proposition \ref{stabilityoper}) and by proving a
semistable reduction theorem for nilpotent opers (Proposition
\ref{ssreduction}). We note that the finiteness result for
indigenous bundles was proved by S. Mochizuki in the rank two case
(see \cite{Mo}) by a different method which does not seem (at the
moment) to lend itself to generalization to higher ranks as it uses
some rather non-linear properties of $p$-curvature and some
peculiarities of rank $2$. The fibre over zero of the
Hitchin-Mochizuki morphism consists of nilpotent opers --- this is the
analogue of the global nilpotent cone of the usual Hitchin morphism.
Since dormant opers (i.e. opers with $p$-curvature zero) are
nilpotent, we deduce from Theorem~\ref{theorem2} that the scheme of
dormant opers with fixed determinant is finite. In the rank two,
genus two case it is possible to compute by various methods the
length of this scheme (see \cite{Mo2}, \cite{lange-pauly08} or
\cite{O2}). Our finiteness result Theorem~\ref{theorem2} naturally
raises the problem of counting the nilpotent and dormant opers for
general $r,p$ and $g$. In \cite{Mo} Mochizuki also showed among
other results that $\mathrm{Nilp}_2(X)$ is a scheme of length
$p^{\dim \Op_{\PGL(2)}}$ for any genus $g$. The analogous result for nilpotent opers of 
higher ranks has been proved recently in \cite{chen}.

The key technical tool in the proof of  Theorem~\ref{theorem2} is
the following result which controls the Harder-Narasimhan polygon of
Frobenius-destabilized bundles (also see
Theorem~\ref{thmoperdominant}).

\begin{thm}\label{theorem1}
Let $(V, \nabla)$ be a semistable local system of degree $0$ and
rank $r$ over the curve $X$. Let $V_\mydot^{HN}$ denote the
Harder-Narasimhan filtration of $V$. Then
\begin{enumerate}
\item
the Harder-Narasimhan polygon $\sP_V$ of $V$ is on or below the
oper-polygon $\sP_r^{oper}$.
\item  we have equality $\sP_V = \sP_r^{oper}$  if and only if $(V,\nabla,V_\mydot^{HN})$
is an oper.
\end{enumerate}
\end{thm}

To better appreciate this result, it is best to put it  in proper
perspective. We note that bundles with connections of $p$-curvature
zero are crystals (in the sense of Grothendieck --- though not
$F$-crystals) and this theorem is, in a sense, an analog of Mazur's
theorem (``Katz' conjecture'') on Hodge and Newton polygons of
$F$-crystals see \cite{mazur73} --- the oper polygon has integer slopes
and plays the role of the Hodge polygon (and if we apply a 180
degree counter clockwise rotation to the polygon plane, then our
result says that the rotated Harder-Narasimhan polygon lies on or
above the oper-polygon and both have the same endpoints).

As in the case of Mazur's Theorem (see \cite{mazur73}), our result is equivalent to a
list of inequalities for the slopes of the graded pieces of the
Harder-Narasimhan filtration.  We establish the result by proving
that the relevant list of inequalities hold using  induction and
by combining and refining some well-known inequalities (see
\cite{shepherd}, \cite{sun99} or \cite{laszlo-pauly}) on the slopes
of the successive quotients of $V_\mydot^{HN}$. As a particular case
of Theorem~\ref{theorem1} we obtain that, when the semistable bundle
$E$ varies, the Harder-Narasimhan polygon of $F^*(E)$ is maximal if
and only if $(F^*(E), \nabla^{can}, F^*(E)_\mydot^{HN})$ is a
dormant oper.

We now turn to describing our results on the Frobenius instability
locus in the moduli space of semi-stable vector bundles on curves.
As pointed out above, the problem of describing the
Frobenius-destabilized locus is one of the least understood aspects of the theory
of vector bundles on curves.  Our next result, provides a
construction of all Frobenius destabilized stable bundles. We state
the result in its weakest form:
\begin{thm} \label{QuotFrob-weak}
Let $X$ be a smooth, projective curve of genus $g\geq 2$ over an
algebraically closed field $k$ of characteristic $p>0$. If $p >
\crg$, then we have
\begin{enumerate}
\item Every stable and Frobenius-destabilized vector bundle $V$ of rank
$r$ and slope $\mu(V)=\mu$ over $X$ is a subsheaf $V\hookrightarrow
F_*(Q)$ for some stable vector bundle $Q$ of rank  $\rk(Q)<r$ and
$\mu(Q) < p\mu$.
\item Conversely, given a semistable vector bundle $Q$ with $\rk(Q)<r$ and $\mu(Q) < p\mu$, every
subsheaf $V\hookrightarrow F_*(Q)$ of rank $\rk(V)=r$ and slope
$\mu(V)=\mu$ is semistable and destabilized by Frobenius.
\end{enumerate}
\end{thm}

It should be remarked that (especially after \cite{JRXY} and
\cite{lange-pauly08}) it has been expected that  every Frobenius-destabilized stable
bundle $V$ should occur as a subsheaf of $F_*(Q)$ for a suitably
chosen $Q$. The main difficulty in realizing this expectation was
the injectivity $V\hookrightarrow F_*(Q)$ for the canonical choice
of $Q$ (the one given by the minimal slope quotient of the
instability flag of $F^*(V)$). We are able to resolve this
difficulty, and it should be remarked that opers play a critical
role in this resolution as well, by providing a uniform bound on the
slopes of subbundles of $F_*(Q)$ (for any stable $Q$) (see
Proposition~\ref{slopesubbundlesFQ}). This bound is finer than the
bounds which exists in literature (see
\cite{lange-pauly08},\cite{JRXY} and \cite{joshi04}). Our method is
to use a method of X. Sun (see \cite{sun06}), especially an
important formula (see eq~\ref{formulasun}) due to him, relating the
slope of a subbundle of $W\subset F_*(Q)$ to the slopes of the
graded pieces of the filtration induced by the oper filtration. We
show that Sun's formula reduces the problem of obtaining a bound to
a numerical optimization problem. We are able to carry out this
optimization effectively (for $p$ bigger than an explicit constant
depending on the degree and the rank, see the statement of the
theorem for details). The final step is to show that every stable
Frobenius-destabilized bundle $V$ is, in fact, a subsheaf of
$F_*(Q)$ for a $Q$ of some rank $\leq\rk(V)-1$ and $\deg(Q)=-1$.
This is the optimal degree allowed from Harder-Narasimhan flag slope
considerations and this is accomplished in
Theorem~\ref{QuotFrobdegree0}.

This paves the way for describing the instability locus $\cJ(r)$ in
terms of a suitable Quot-scheme. We show that under the assumption
that $p > \crg$,
there is a morphism  $\pi$ (see
Theorem~\ref{locusJ} for the precise statement) from a disjoint union of
relative Quot-schemes to $\cJ(r)$, which surjects onto the stable
part of the instability locus $\cJ^s(r)$.  We expect that $\pi$ is
birational over some irreducible component of $\cJ(r)$, but we are
not able to show that for $r > 2$ and we hope to return to it at a
later work. We also show (Proposition \ref{bijFrobopersQuot}) that
the set of dormant opers equals a relative Quot-scheme
$$\alpha : \QQuot(r,0) \rightarrow  \Pic^{-(r-1)(g-1)}(X)$$
over the Picard variety of line bundles over $X$ of degree
$-(r-1)(g-1)$ such that the fiber $\alpha^{-1}(Q)$ over $Q \in
\Pic^{-(r-1)(g-1)}(X)$ equals the Quot-scheme
$\mathrm{Quot}^{r,0}(F_*(Q))$ parameterizing subsheaves of degree
$0$ and rank $r$ of the direct image $F_*(Q)$. In particular, since
$\mathrm{Quot}^{r,0}(F_*(Q)) \not= \emptyset$ (Proposition
\ref{nonemptyQuot}), we obtain the existence of dormant opers.

As mentioned earlier, the instability locus $\cJ(r)$, and especially
$\cJ^s(r)$, is equipped with a stratification, and one would like a
concrete description of the instability strata. This is at the
moment only understood completely for $p=2,r=2$ (\cite{JRXY}),
though there are partial results for $p\leq3,g\leq 3,r\leq 3$, or
$p\geq 2g, r\leq 2$, in the references cited at the end of the first
paragraph of this paper. Our first result (Theorem~\ref{theorem2})
identifies the minimal dimensional stratum of this stratification on
$\cJ^s(r)$. The finiteness result Theorem (\ref{theorem2}) provides
the following description of the minimal dimensional stratum which
corresponds to the highest Harder-Narasimhan polygon: the stratum
corresponding to the oper-polygon is zero dimensional. We call this
the \emph{operatic locus}. The operatic locus is characterized as
the locus of semi-stable bundles whose Frobenius pull-back and its
Harder-Narasimhan filtration together with its canonical connection
and is a dormant oper. The finiteness result (Theorem~\ref{theorem2})
shows that the operatic locus is finite (Theorem~\ref{opersquot}). As
applications of the relationship between opers and Quot-schemes, we
mention the following results. Firstly, for rank two, we show
(Theorem~\ref{dimJ2}) that any irreducible component of $\cJ(2)$
containing a dormant oper has dimension $3g-4$, which completes some
results on $\dim \cJ(2)$ for general curves due to S. Mochizuki.
Secondly, we deduce from Theorem~\ref{theorem2} that $\dim
\mathrm{Quot}^{r,0}(F_*(Q)) = 0$ (its expected dimension) for any
line bundle $Q$ of degree $-(r-1)(g-1)$.

\subsection{Acknowledgements}
The first author wishes to thank the University of Montpellier II for financial support
of a research visit during spring 2008. We would like to thank S. Brochard, M. Gr\"ochenig, P. Newstead, V. Mehta and
B. To\"en for helpful discussions. We thank the referee for many corrections and suggestions which have improved the readability of this paper. We also thank R. Bezrukavnikov for informing us of another approach to prove 
finiteness of the Hitchin-Mochizuki morphism. This approach has been realized in the preprint \cite{chen} (2013).

\bigskip

In the early 1990s Vikram B. Mehta suggested the problem of studying
Frobenius destabilized bundles to the first author and encouraged him to
think about it. Both authors acknowledge Vikram's mathematical
tutelage, of many years, with great pleasure and deepest gratitude.

\newpage
\tableofcontents

\section{Generalities}
\subsection{The Frobenius morphism and $p$-curvature}
\subsubsection{Definitions}
Let $X$ be a smooth projective curve of genus $g \geq 2$ defined over an algebraically
closed field $k$ of characteristic $p > 0$. Let $F : X \ra X$ be the absolute
Frobenius morphism  of the curve $X$.
One also has at one's disposal the relative Frobenius morphism $X\to X^{(1)}=X\times_{k,\sigma}k$,
where $\sigma: k \rightarrow k$ denotes the Frobenius of the base field $k$.
For notational simplicity we will work with the absolute Frobenius morphism as opposed to the relative Frobenius morphism.  As we work over an algebraically closed field, one has an equivalence of categories of coherent sheaves on $X$ and $X^{(1)}$ induced by the natural 
morphism $X^{(1)}\to X$, and  the distinction between relative and absolute Frobenius is immaterial for our purposes. Occasionally 
we will (when we make relative constructions) invoke the relative Frobenius morphism and we will say so. We hope that the reader will 
not find the change too distressing (when it is needed). The full yoga of the two morphisms is spelled out in many places 
(for instance in \cite{Katz}). We denote by $\Omega_X^{1}$ and $T_X$ the canonical bundle and the tangent bundle of the curve $X$.

Given a local system $(V, \nabla)$ over $X$, i.e. a pair $(V, \nabla)$ consisting of a
vector bundle $V$ over $X$ and a connection $\nabla$ on $V$, we associate (see e.g. \cite{Katz})
the $p$-curvature morphism
$$ \psi(V, \nabla) : T_X \lra \End(V), \qquad D \mapsto \nabla(D)^p - \nabla(D^p).$$
Here $D$ denotes a local vector field, $D^p$ its $p$-th power and $\End(V)$ denotes the
sheaf of $\cO_X$-linear endomorphisms of $V$.

\subsubsection{Properties of the $p$-curvature morphism }
By \cite{Katz} Proposition 5.2 the $p$-curvature morphism  is
$\cO_X$-semi-linear, which means that it corresponds to an
$\cO_X$-linear map, also denoted by $\psi(V, \nabla)$
\begin{equation} \label{pcurvature}
 \psi(V,\nabla) : F^* T_X \longrightarrow \End(V).
\end{equation}
In the sequel we will always consider $p$-curvature maps as $\cO_X$-linear maps
\eqref{pcurvature}.
Given local systems $(V, \nabla_V)$ and $(W, \nabla_W)$ over $X$, we can naturally
associate the local systems given by their tensor product $(V \otimes W, \nabla_V \otimes
\nabla_W)$ and the determinant $(\det V, \det \nabla_V)$.

\begin{prop} \label{proppcurv}
We have the following relations on the $p$-curvature maps
\begin{itemize}
\item[(i)] $\psi(\cO_X, d) = 0$, where $d : \cO_X \ra \Omega_X^1$ is the differentiation operator.
\item[(ii)] $\psi(V \otimes W, \nabla_V \otimes \nabla_W) = \psi(V, \nabla_V) \otimes \Id_W +
\Id_V \otimes \psi(W, \nabla_W)$.
\item[(iii)] $\psi(\det V, \det \nabla_V)  = \mathrm{Tr} \circ \psi(V, \nabla_V)$, where $\mathrm{Tr} : \End(V) \lra \cO_X$ denotes the trace map.
\end{itemize}
\end{prop}

\begin{proof}
Part (i) is trivial.
Part (ii) follows from the definition of the $p$-curvature and the fact that the
two terms of $\nabla_{V \otimes W} (D) = \nabla_V(D) \otimes  \Id_W +
\Id_V \otimes \nabla_W(D)$ commute in the ring $\End_k(V)$ of $k$-linear
endomorphisms of $V$. As for part (iii), we iterate the formula proved in (ii) in order
to obtain an expression for the $p$-curvature $\psi(V^{\otimes r}, \nabla_{V^{\otimes r}})$
with $r = \rk(V)$ and show that the $F^* \Omega_X^1$-valued endomorphism
$\psi(V^{\otimes r}, \nabla_{V^{\otimes r}})$ of $V^{\otimes r}$
induces an endomorphism on the quotient $V^{\otimes r} \lra \Lambda^r V = \det V$. The formula in (iii)
then follows by taking a basis of local sections of $V$ and comparing both sides of the
equality.
\end{proof}

Let $(L, \nabla_L)$ be a local system of rank $1$, i.e. $L$ is a line bundle.
Let $r$ be an integer not divisible by $p$. We say that
$(L, \nabla_L)$  is an $r$-torsion local system if $(L^{\otimes r}, \nabla_{L^{\otimes r}}) =
(\cO_X, d)$. Note that for an $r$-torsion line bundle $L$ there exists a unique
connection $\nabla_L$ on $L$ such that the local system $(L, \nabla_L)$ is
$r$-torsion.

\begin{prop} \label{pcurvtorsion}
Let $(L, \nabla_L)$ be an $r$-torsion local system. If $p$ does not divide $r$, then
$$ \psi (L, \nabla_L) = 0.$$
\end{prop}

\begin{proof}
By Proposition \ref{proppcurv} (i) and (ii) we know that
$$ 0 = \psi(\cO_X, d) = \psi(L^{\otimes r}, \nabla_{L^{\otimes r}}) =
r \psi(L, \nabla_L).$$
Hence, if $p$ does not divide $r$, we obtain the result.
\end{proof}

\subsubsection{Cartier's theorem}
We now state a fundamental property of the $p$-curvature saying
roughly that the tensor $\psi(V, \nabla)$ is the obstruction to
descent of the bundle $V$ under the Frobenius map $F : X \ra X$.

\begin{thm}[\cite{Katz} Theorem 5.1] \label{Cartierthm}
\begin{enumerate}
\item Let $E$ be a vector bundle over $X$. The pull-back $F^*(E)$ under the Frobenius morphism
carries a canonical connection $\nabla^{can}$, which satisfies the equality
$\psi(F^*(E),\nabla^{can}) = 0$.
\item Given a local system $(V, \nabla)$ over $X$, there exists a vector bundle $E$ such that
$(V, \nabla) = (F^*(E),\nabla^{can})$ if and only if $\psi(V, \nabla) = 0$.
\end{enumerate}
\end{thm}

\subsection{Direct images under the Frobenius morphism }
Applying the Grothendieck-Riemann-Roch formula to the morphism $F: X
\rightarrow X$ we obtain the following  useful formula for the
degree of the direct image of a vector bundle under the Frobenius
morphism
\begin{lem} \label{degFQ}
Let $Q$ be a vector bundle of rank $q = \rk(Q)$ over $X$. Then we have
$$ \deg(F_*(Q)) = \deg(Q) + q(p-1)(g-1), \qquad \text{and} \qquad \mu(F_*(Q)) = \frac{\mu(Q)}{p} + \left( 1 - \frac{1}{p} \right)(g-1). $$
\end{lem}


\subsection{The Hirschowitz bound and Quot-schemes}
\subsubsection{Existence of subbundles of given rank and degree}
We will use the following result due to A. Hirschowitz \cite{hir}
(see also \cite{lange}).
\begin{thm}\label{hirschowitz-bound}
Let $X$ be a smooth, projective curve of genus $g\geq 2$. Let $V$ be
a vector bundle of rank $n$ and degree $d$ over $X$. Let $m$ be an integer satisfying $1 \leq
m \leq n-1$. Then there exists a subbundle $W\subset V$
of rank $m$ such that
$$\mu(W)\geq
\mu(V)-\left(\frac{n-m}{n}\right)(g-1)-\frac{\varepsilon}{mn},$$
where $\varepsilon$ is the unique integer satisfying $0\leq \varepsilon\leq n-1$ and
$$\varepsilon+m(n-m)(g-1)\equiv md\mod{n}.$$
\end{thm}

\subsubsection{Non-emptiness of Quot-schemes}
Let $Q$ be a vector bundle of rank $q$ and let $r$ be an integer
satisfying $q < r < pq$. We denote by
$$ \mathrm{Quot}^{r,0}(F_*(Q)) $$
the Quot-scheme parameterizing rank-$r$ subsheaves of degree $0$ of
the vector bundle $F_*(Q)$.


The following result is an immediate consequence of Theorem \ref{hirschowitz-bound}.
\begin{prop} \label{nonemptyQuot}
If $\deg(Q) \geq -(r-q)(g-1)$,
then $ \mathrm{Quot}^{r,0}(F_*(Q)) \not= \emptyset$.
\end{prop}

\begin{proof}
By Theorem \ref{hirschowitz-bound} there exists
a subsheaf $W\subset F_*(Q)$ with $\rk(W)=r$ such that
\begin{equation}\label{hir}
\mu(W)\geq
\mu(F_*(Q))-\left(\frac{pq-r}{pq}\right)(g-1)-\frac{\varepsilon}{pqr},
\end{equation}
where $\epsilon$ is the unique integer satisfying $0\leq \varepsilon\leq pq-1$ and
$$\varepsilon+r(pq-r)(g-1)=r\deg(F_*(Q))\mod{pq}.$$
Now by Lemma \ref{degFQ} we have $\deg(F_*(Q))= \deg(Q) + q(p-1)(g-1)$, so that
\begin{eqnarray*}
  \varepsilon &=& r(\deg(Q) + q(p-1)(g-1)) - r(pq-r)(g-1)\mod{pq},  \\
   &=&r \deg(Q) +pqr(g-1)-rq(g-1)-pqr(g-1)+r^2(g-1)\mod{pq}  \\
   &=&r \deg(Q) -rq(g-1)+r^2(g-1)\mod{pq}  \\
   &=&r \left[ \deg(Q) + (r-q)(g-1) \right] \mod{pq}
\end{eqnarray*}

First we consider the case when $r \left[ \deg(Q) + (r-q)(g-1) \right] \leq pq -1$. We
replace $\varepsilon$ with $r \left[ \deg(Q) + (r-q)(g-1) \right]$ in
the inequality \eqref{hir} and we obtain $\mu(W) \geq 0$.
Thus we can find a subsheaf $W$ of rank $r$ of $F_*(Q)$ with
$\deg(W) \geq 0$. If $\deg(W)>0$, then we take a lower modification $V$ of
$W$ to get a subsheaf with $\deg(V)=0$ and $\rk(V)=r$.


Secondly, we consider the case when $r \left[ \deg(Q) + (r-q)(g-1) \right] \geq pq$ or, equivalently
$\deg(Q) \geq  \frac{pq}{r} - (r-q)(g-1)$. When we combine this inequality and
$\frac{- \varepsilon}{pqr} \geq - \frac{1}{r}$, with inequality \eqref{hir} we obtain the lower
bound $\mu(W) \geq 0$,
and we can conclude as above.
\end{proof}

\subsubsection{Dimension estimates for the Quot-scheme}
We recall the following result \cite[Expos\'e IV, Cor. 5.2]{Gro} (or \cite[Cor. 2.2.9]{huy-book}), which will be used in
Section~\ref{dimensions-rank-two} in the rank two case.

\begin{prop} \label{dimestimateQuot}
Any irreducible
component of $\mathrm{Quot}^{r,0}(F_*(Q))$  has dimension at least $$r \left[ (r-q)(g-1) + \deg(Q) \right].$$
\end{prop}

\begin{proof}
By \cite[Expos\'e IV, Cor. 5.2]{Gro} (or \cite[Cor. 2.2.9]{huy-book}) the dimension at a point $E \in \mathrm{Quot}^{r,0}(F_*(Q))$ is at least
$\chi(X, \Hom( E, F_*(Q)/E))$, which is easily computed using Lemma \ref{degFQ}.
\end{proof}


\section{Opers}
\subsection{Definition and examples of opers and dormant opers}
Opers were introduced by A. Beilinson and V. Drinfeld in \cite{BD1} (see also \cite{BD2}).
In this paper, following \cite[Section 5]{JRXY}, we slightly modify the definition (in \cite{BD1,BD2}).

\begin{defi} \label{defioper}
An {\em oper} over a smooth algebraic curve $X$ defined over an algebraically closed field $k$ of
characteristic $p>0$ is a triple $(V, \nabla, V_\mydot)$, where
\begin{enumerate}
\item  $V$ is a vector bundle over $X$,
\item  $\nabla$ is a connection on $V$,
\item  $V_\mydot : 0 =  V_l  \subset V_{l-1} \subset \cdots \subset V_1 \subset V_0 = V$ is
a decreasing filtration by subbundles of $V$, called the oper flag.
\end{enumerate}
These data have to satisfy the following conditions
\begin{enumerate}
\item $\nabla(V_i) \subset V_{i-1} \otimes \Omega_X^1$ for $1 \leq i \leq l-1$,
\item the induced maps $\left( V_i/ V_{i+1} \right) \stackrel{\nabla}{\longrightarrow} \left( V_{i-1}/V_i
\right) \otimes \Omega_X^1$ are
isomorphisms for $1 \leq i \leq l-1$.
\end{enumerate}
\end{defi}

\begin{defi} \label{defiFroboper}
We say that an oper $(V, \nabla, V_\mydot)$ is a {\em dormant oper} if $\psi(V,\nabla) = 0$,
i.e. by Cartier's theorem, if the oper $(V, \nabla, V_\mydot)$ is of the form
$(F^*(E), \nabla^{can}, V_\mydot)$.
\end{defi}

Given an oper $(V, \nabla, V_\mydot)$ we denote by $Q$ the first quotient of $V_\mydot$, i.e.,
$$ Q = V_0/ V_1.$$
We define the {\em degree}, {\em type} and {\em length} of an oper $(V,\nabla, V_\mydot)$ by
$$ \deg (V,\nabla, V_\mydot) = \deg (V), \qquad \mathrm{type}(V,\nabla, V_\mydot) = \rk (Q) , \qquad
\mathrm{length}(V,\nabla, V_\mydot) = l.$$
The following formulae are immediately deduced from the definition:
\begin{equation} \label{quotientoper}
V_i/ V_{i+1} \cong Q \otimes (\Omega^1_X)^{\otimes i}  \qquad \text{for} \ 0 \leq i \leq l-1,
\end{equation}
\begin{equation} \label{degoper}
 \deg (V) = l( \deg (Q) + \rk(Q)(l-1)(g-1) ), \qquad \rk(V) = \rk(Q) l.
\end{equation}
We note that the existence of a connection $\nabla$ on $V$ implies that $\deg(V)$ is
divisible by $p$ (see \cite{Katz}) --- the connection $\nabla$ can have non-zero $p$-curvature. 
Note that our notion of oper of type $1$ corresponds to the notion
of $\GL(l)$-oper in \cite{BD1}.

\begin{rem} \label{flagHN}
If the quotient $Q$ is semistable, then the oper flag $V_\mydot$ coincides with the
Harder-Narasimhan filtration of the vector bundle $V$. To see this we note that the Harder-Narasimhan filtration is the unique
filtration characterized by the two properties (see section \ref{HNdef}): (1) graded pieces are semistable and 
(2) the slopes are strictly decreasing (starting with $\mu_{max}$). Now as $Q$ is semistable it suffices 
to check that slopes of the graded pieces satisfy the second property. This is an easy numerical check 
which is immediate from the formulae for the graded pieces given above.
\end{rem}

The following result which combines the results of \cite[Section 5]{JRXY} and \cite{sun06} provides
the basic example of opers in characteristic $p>0$.
\begin{thm} \label{exampleoper}
Let $E$ be any vector bundle over $X$ and let $F: X \rightarrow X$ be the absolute
Frobenius of $X$. Then the triple
$$(V = F^*(F_* (E)) , \nabla^{can} , V_\mydot ),$$
where $V_\mydot$ is the canonical filtration defined in \cite{JRXY}
section  5.3, is a {\em dormant oper} of type $\rk(E)$ and length $p$.
Moreover, there is an equality $V_0/V_1 = Q = E$.
\end{thm}


\subsection{Recollection of results on opers}
In this section we concentrate on opers of type $1$, which we call
opers for simplicity. Let $r \geq 2$ be an integer. Although all the results which 
are mentioned in this section are proved in \cite{BD1} and \cite{BD2} over the complex
numbers, we recall the main facts and include some
proofs for the convenience of the reader. 

\bigskip

We assume that $p >r$ and that $p$ does not divide $g-1$.

\bigskip

We recall that a $\GL(r)$-oper (as defined by \cite{BD1}) corresponds to an oper $(V, \nabla, V_\mydot)$ of type $1$ and length $r$ as defined by
\ref{defioper}. An $\SL(r)$-oper is a $\GL(r)$-oper $(V, \nabla, V_\mydot)$ which satisfies 
$(\det V, \det \nabla) \map{\sim} (\cO_X, d)$, where $d: \cO_X \ra \Omega_X^1$ is the differentiation operator on $X$.

\bigskip

Given a theta-characteristic $\theta$ on $X$ we consider the (unique up to isomorphism) rank-$2$ vector bundle 
$\cG_2(\theta)$ defined as the non-split extension of $\theta^{-1}$ by $\theta$
$$ 0  \lra \theta \lra \cG_2(\theta) \lra \theta^{-1} \lra 0, $$
and equip it with the filtration $\cG_2(\theta)_1 = \theta$.
Under the above assumption on the characteristic $p$, any connection $\nabla$ on $\cG_2(\theta)$ defines a $\GL(2)$-oper 
$(\cG_2(\theta), \nabla, \cG_2(\theta)_\mydot)$. For arbitrary $r$, we introduce the rank-$r$ vector bundle given by the symmetric power  
$$\cG_r(\theta) := \Sym^{r-1}(\cG_2(\theta))$$ 
and equip it with the filtration $\cG_r(\theta)_\mydot$ given by the natural subbundles
$$\cG_r(\theta)_i = \theta^{i} \otimes \cG_{r-i}(\theta) \qquad  \text{for} \qquad 1 \leq i \leq r-1.$$
We define $\cG_1(\theta) = \cO_X$. As for $r= 2$ any connection $\nabla$ on $\cG_r(\theta)$ defines a
$\GL(r)$-oper $(\cG_r(\theta) , \nabla, \cG_r(\theta)_\mydot)$. Conversely, it is easy to show that, up to tensor product by a line
bundle, the underlying filtered bundle $(V, V_\mydot)$ of a $\GL(r)$-oper $(V, \nabla, V_\mydot)$ is 
isomorphic to $(\cG_r(\theta) , \cG_r(\theta)_\mydot)$. It is straightforward to establish that, if $r$ is odd $\cG_r(\theta)$ does not
depend on $\theta$ and, if $r$ is even $\cG_r(\theta \otimes \gamma) = \cG_r(\theta) \otimes \gamma$ for any $2$-torsion line bundle
$\gamma  \in JX[2]$. We recall that $\det \cG_r(\theta) = \cO_X$ and that $\cG_r(\theta)_\mydot$ is the Harder-Narasimhan filtration 
of $\cG_r(\theta)$.

\bigskip

We also note that the space of global sections
\begin{equation} \label{sectionsGtheta}
 \mathrm{Hom}(\theta^{-(r-1)}, \theta \otimes \cG_{r-1}(\theta) \otimes \Omega_X^1) = H^0(X, \cG_{r-1}(\theta) \otimes \theta^{r+2} )
\end{equation}
has a canonical filtration induced by the filtration $\cG_{r-1}(\theta)_\mydot$ and its associated graded vector space is given by
$$ {\bf W}_r = \bigoplus_{i=2}^r H^0(X, (\Omega_X^1)^{\otimes i}).$$
We recall that $\dim {\bf W}_r = (g-1)(r^2-1)$.

\bigskip

In order to deal with moduli of opers we introduce the following moduli stacks:

\begin{enumerate}
\bigskip

\item the stack $\Loc_{\GL(r)}$ parameterizing rank-$r$ local systems over $X$ defined as the sheaf of groupoids which
assigns to a $k$-scheme $S$ the groupoid $\Loc_{\GL(r)}(S)$ of pairs $(\cV, \nabla)$ where $\cV$ is a vector
bundle over $X \times S$ having degree zero on the fibers over $S$ and 
$\nabla : \cV \ra \cV \otimes p_X^* \Omega_X^1$  is an $\cO_S$-linear map satisfying Leibniz rule.
Morphisms in $\Loc_{\GL(r)}(S)$ are isomorphisms of vector bundles $\cV \map{\sim} \cV'$ commuting with the connections. The stack $\Loc_{\GL(r)}$
is algebraic (see e.g. \cite{laszlo-pauly} Corollary 3.1).

\bigskip

\item the stack $\Loc_{\SL(r)}$ parameterizing rank-$r$ local systems over $X$ with trivial determinant defined as the fiber
over the trivial local system $(\cO_X,d)$ of the map of stacks $\Loc_{\GL(r)} \ra \Loc_{\GL(1)}$ induced by the determinant. More 
explicitly, for a $k$-scheme $S$ the groupoid $\Loc_{\SL(r)}(S)$ consists of triples $(\cV, \nabla, \alpha)$, where 
$\cV$ and $\nabla$ are as in (1) and $\alpha : (\det \cV, \det \nabla) \map{\sim} (\cO_{X \times S} , d)$ is an
isomorphism of local systems. Here $d : \cO_{X \times S} \ra p_X^* \Omega_X^1$ is the $\cO_S$-linear differentiation operator
of $X \times S/S$. Since the morphism of stacks  $pt \ra \Loc_{\GL(1)}$ corresponding to the trivial local system
$(\cO_X, d)$ is representable, the morphism $\Loc_{\SL(r)} \ra
\Loc_{\GL(r)}$ is also representable (by base change) and therefore $\Loc_{\SL(r)}$ is an algebraic stack.

\bigskip

\item the stack $\Op_{\SL(r)}$ parameterizing $\SL(r)$-opers over $X$ is the stack which assigns to a $k$-scheme $S$ the
groupoid $\Op_{\SL(r)}(S)$ consisting of quadruples  $(\cV, \nabla, \cV_\mydot, \alpha)$, where $\cV, \nabla, \alpha$ are as
in (1) and (2), and $\cV_\mydot$ denotes a decreasing filtration of subbundles of $\cV$. Moreover these data have to satisfy the
conditions (1) and (2) of Definition \ref{defioper} over the relative curve  $X \times S \rightarrow S$. Morphisms in $\Op_{\SL(r)}(S)$ are isomorphisms
of vector bundles commuting with the connections, the trivializations of the determinant and preserving the filtrations.
The  morphism which forgets the filtration $\Op_{\SL(r)} \ra \Loc_{\SL(r)}$ is an immersion, since the filtration $\cV_\mydot$ coincides 
with the Harder-Narasimhan filtration on $\cV$ (see Remark \ref{flagHN} and e.g. \cite{heinloth} Proposition 5.3). 

\bigskip

\item the stack $\mathcal{T}_r$  parameterizing $r$-torsion line bundles over $X$ is the stack which assigns to a $k$-scheme $S$ the 
groupoid $\mathcal{T}_r(S)$ consisting of pairs $(\cL, \beta)$ where $\cL$ is a line bundle over $X \times S$ and $\beta : \cL^{\otimes r}
\map{\sim} \cO_{X \times S}$ is an isomorphism. Morphisms in $\mathcal{T}_r(S)$ are isomorphisms between line bundles  commuting with the 
trivializations, hence the automorphism group $\mathrm{Aut}(\mathcal{T}_r(S)) = \mu_r$. Note that $\mathcal{T}_r$ is the fiber product of the
stack morphism $[r] : \mathcal{P}ic \ra \mathcal{P}ic$ between Picard stacks of the curve $X$ induced by the $r$-th power map and the 
stack morphism $pt \rightarrow \mathcal{P}ic$ corresponding to the trivial line bundle. 

\bigskip

\item the group $\mathrm{Aut}^0(\cG_r(\theta))$ of trivial-determinant automorphisms of the bundle 
$\cG_r(\theta)$ acts via conjugation on the space of connections of $\cG_r(\theta)$. We denote this
action by $\nabla \mapsto \nabla^\phi$ for $\phi \in \mathrm{Aut}^0(\cG_r(\theta))$. We introduce the 
space $\mathrm{Conn}^0(\cG_r(\theta))$ parameterizing gauge equivalence classes of connections on $\cG_r(\theta)$ with trivial determinant. Then
we have

\begin{lem} \label{operaffine}
The space  $\mathrm{Conn}^0(\cG_r(\theta))$ is an affine space whose underlying vector space equals the space of 
sections $H^0(X, \cG_{r-1}(\theta) \otimes \theta^{r+2})$ introduced
in \ref{sectionsGtheta}. Note that $H^0(X, \cG_{r-1}(\theta) \otimes \theta^{r+2})$ is naturally a subspace of 
$H^0(X, \End_0(\cG_r(\theta)) \otimes \Omega_X^1)$.
\end{lem}

\begin{proof}
It can be shown that, fixing a connection $\nabla_0$ on $\cG_r(\theta)$, for any connection $\nabla$ there exists a 
unique (up to homothety) $\phi \in  \mathrm{Aut}^0(\cG_r(\theta))$ and a unique $\omega \in  H^0(X, \cG_{r-1}(\theta) \otimes \theta^{r+2})$ such
that 
$$ \nabla^\phi = \nabla_0 + \omega.$$
\end{proof}

\noindent
In other words, fixing a connection $\nabla_0$ on $\cG_r(\theta)$, the affine subspace 
$$\{ \nabla_0 + \omega \ | \ \omega \in H^0(X, \cG_{r-1}(\theta) \otimes \theta^{r+2}) \}$$ 
of the space of connections on $\cG_r(\theta)$ is
isomorphic to the quotient $\mathrm{Conn}^0(\cG_r(\theta))$. In particular, we have a transitive and free action of $H^0(X, \cG_{r-1}(\theta) \otimes \theta^{r+2})$ on the space
$\mathrm{Conn}^0(\cG_r(\theta))$ --- note that any $\omega \in H^0(X, \cG_{r-1}(\theta) \otimes \theta^{r+2})$
is fixed under conjugation by $\mathrm{Aut}^0(\cG_r(\theta))$.

\end{enumerate}

\bigskip

The previous moduli stacks are related as follows. 

\begin{prop}[\cite{BD1} Proposition 3.4.2] \label{isooperslr}
Fixing a theta-characteristic $\theta$ we have an isomorphism between
stacks 
$$ \tau_\theta : \Op_{\SL(r)} \map{\sim} \mathrm{Conn}^0(\cG_r(\theta)) \times \mathcal{T}_r. $$
\end{prop}

\begin{proof}
Over a $k$-scheme $S$ an object in $\Op_{\SL(r)}(S)$ is a quadruple $(\cV, \nabla, \cV_\mydot, \alpha)$ described above. We consider the quotient
line bundle $\mathcal{Q} = \cV_0 / \cV_1$ over $S \times X$. Then, using the connection $\nabla$ and its induced isomorphisms on the quotients
$\cV_i / \cV_{i+1}$, we obtain an isomorphim $\det \cV \map{\sim}  \mathcal{Q}^{\otimes r} \otimes (\Omega_X^1)^{\otimes \frac{r(r-1)}{2}}$.
We put $\cL = \mathcal{Q} \otimes p_X^* \theta^{\otimes r -1}$. Composing the previous isomorphism with $\alpha$, we obtain an 
isomorphism $\beta : \cL^{\otimes r} \map{\sim} \cO_{X \times S}$. We recall that the line bundle $\cL$ has a unique 
$\cO_S$-linear connection $\nabla_{\cL}$ such that $(\cL, \nabla_{\cL}) ^{\otimes r} = (\cO_{X \times S}, d)$. Tensoring 
$(\cV, \nabla)$ with $(\cL^{-1}, \nabla_{\cL^{-1}})$ we obtain a local system  $(\cW, \nabla_{\cW}, \cW_\mydot) = 
(\cV \otimes \cL^{-1}, \nabla \otimes \nabla_{\cL^{-1}} ,  \cV_\mydot \otimes \cL^{-1})$, which clearly satisfies the oper
conditions (1) and (2). By restriction of $\cW$ to a fiber $\{ x \} \times S$, we show that the induced isomorphisms between quotients of the
filtration by the $\cO_S$-linear connection $\nabla_\cW$ give a splitting of the filtration $\cW_{\mydot | \{ x \} \times S}$, hence
an isomorphism $\cW_{| \{ x \} \times S} \cong \cO_S^{\oplus r }$ for any $x \in X$. This implies that 
$\cW \map{\sim} p_X^* \cG_r(\theta)$ and therefore the connection $\nabla_\cW$ induces a map $\phi : S \ra \mathrm{Conn}^0(\cG_r(\theta))$, i.e.
an object of $\mathrm{Conn}^0(\cG_r(\theta)) (S)$. The map $\tau_\theta$ then associates to the quadruple $(\cV, \nabla, \cV_\mydot, \alpha)$
the pair $(\phi, (\cL, \beta))$. An inverse map to $\tau_\theta$ is naturally constructed by choosing a connection on $\cG_r(\theta)$.
\end{proof}

\bigskip

In particular, for any theta-characteristic $\theta$ we obtain via the above isomorphism a map
$$ \sigma_\theta : \mathrm{Conn}^0(\cG_r(\theta)) \ra \Op_{\SL(r)}. $$
Note that the affine space $\mathrm{Conn}^0(\cG_r(\theta))$ does not depend on the theta-characteristic
$\theta$, since for any two theta-characteristics $\theta$ and $\theta'$ there exists a 
canonical isomorphism $\mathrm{Conn}^0(\cG_r(\theta)) \cong \mathrm{Conn}^0(\cG_r(\theta'))$.

\bigskip
In \cite{BD1} section 3.1 the authors introduce the notion of $\PGL(r)$-oper and show that the stack of $\PGL(r)$-opers
$\Op_{\PGL(r)}$ is an affine space (\cite{BD1} Proposition 3.1.10). Since our results can be stated in terms of
vector bundles only, we decided to avoid introducing $G$-opers for a general group $G$ in order to keep our paper more
accessible. We refer the interested reader to \cite{BD1} and \cite{BD2}. For completeness, we mention that
the natural group homomorphism $\SL(r) \ra \PGL(r)$ induces a projection
$$ pr:  \Op_{\SL(r)} \lra \Op_{\PGL(r)}.$$
In \cite{BD1} sections 3.4.1 and 3.4.2 the authors construct a section of $pr$. More precisely, they show that
for any theta-characteristic $\theta$, there is an isomorphism
$$ \phi_\theta : \Op_{\PGL(r)} \times \mathcal{T}_r \map{\sim}  \Op_{\SL(r)}. $$
In Proposition \ref{isooperslr} we reformulated the above isomorphism in terms of vector bundles. In 
particular, for any theta-characteristic $\theta$ we obtain an isomorphism
$$ pr \circ \sigma_\theta :  \mathrm{Conn}^0(\cG_r(\theta)) \map{\sim} \Op_{\PGL(r)}, $$
which does not depend on the theta-caracteristic (see \cite{BD1} section 3.4.2). Although we have not
defined the space $\Op_{\PGL(r)}$ in this paper, we shall however use in the sequel $\Op_{\PGL(r)}$ as a
notation for the affine space $\mathrm{Conn}^0(\cG_r(\theta))$.

\bigskip

Although not needed in this paper, we mention that, as a corollary to Proposition \ref{isooperslr}, the group
$JX[r]$ acts on the stack $\Op_{\SL(r)}$ via tensor product with $r$-torsion local systems. The set of connected
components of $\Op_{\SL(r)}$ is in one-to-one correspondence with $JX[r]$.

\bigskip

\subsection{The Hitchin-Mochizuki morphism}

We consider the morphism of stacks 
\begin{equation} \label{charpsi}
\Loc_{\GL(r)} \lra \oplus_{i=1}^r H^0(X, F^*(\Omega_X^1)^{\otimes i}), \qquad
(V, \nabla) \mapsto \mathrm{Char}(\psi(V, \nabla)),
\end{equation}
where $\mathrm{Char}(\psi(V, \nabla))$ denotes the vector in
$\oplus_{i=1}^r H^0(X, F^*(\Omega_X^1)^{\otimes i})$ whose components are given
by the coefficients of the characteristic polynomial of the
$p$-curvature $\psi(V,\nabla) : V \ra V \otimes F^*(\Omega_X^1)$. It is shown in
\cite{laszlo-pauly} Proposition 3.2 that the components  of $\mathrm{Char}(\psi(V, \nabla))$
descend under the Frobenius morphism. This implies that the morphism  \eqref{charpsi} factorizes as
$$ \Loc_{\GL(r)} \map{\Phi} V_r := \oplus_{i=1}^r H^0(X, (\Omega_X^1)^{\otimes i})
\map{F^*} \oplus_{i=1}^r H^0(X, F^*(\Omega_X^1)^{\otimes i}).$$

\begin{prop}
The image of the composition of the natural map $\Loc_{\SL(r)} \ra \Loc_{\GL(r)}$ with $\Phi$ 
is contained in the subspace ${\bf W}_r \subset {\bf V}_r$.
\end{prop}

\begin{proof}
Let $S$ be a $k$-scheme and let $(\cV, \nabla, \alpha)$ be  an object of the groupoid $\Loc_{\SL(r)}(S)$. By the relative version
of Proposition \ref{proppcurv} (i) and (iii), we obtain that $\mathrm{Tr}[ \psi(\cV, \nabla) ] = \psi(\det \cV, \det \nabla) = 0$
in $\Mor ( S, H^0(\Omega_X^1))$.
\end{proof}

We also denote the composite map by $\Phi: \Loc_{\SL(r)} \ra {\bf W}_r$.

\bigskip

Given a theta-characteristic $\theta$ on $X$, we consider the composite morphism
$$ \HM : \Op_{\PGL(r)} \map{\sigma_\theta} \Op_{\SL(r)} \lra \Loc_{\SL(r)} \map{\Phi}
{\bf W}_r,$$
which we call the {\em Hitchin-Mochizuki} morphism. By the above considerations  the
Hitchin-Mochizuki morphism  can be identified with a self-map of the affine space of dimension
$(g-1)(r^2-1)$.

\begin{prop} \label{HMI}
The morphism $\HM$ does not depend on the theta-characteristic $\theta$.
\end{prop}

\begin{proof}
Given two theta-characteristics $\theta$ and $\theta'$ on $X$, we can write
$\theta' = \theta \otimes \alpha$ with $\alpha$ a $2$-torsion line bundle on $X$.
By \cite{BD1} section 3.4.2 there exists an $r$-torsion line bundle $L$ (take $L = \cO_X$ if $r$ is odd and $L = \alpha$ if $r$ is 
even) such that  $\sigma_{\theta'}  = T_L \circ \sigma_\theta$, where
$T_L$ is the automorphism of $\Op_{\SL(r)}(X)$ induced by tensor product with
the $r$-torsion local system $(L, \nabla_L)$. Here $\nabla_L$ is the unique connection on $L$ such that
$(L, \nabla_L)^{\otimes r} = (\cO_X, d)$. Now the result follows because,
by Proposition \ref{proppcurv} (ii) and by Proposition \ref{pcurvtorsion}, we
have for any local system $(V, \nabla_V)$
$$ \psi(V \otimes L, \nabla_V \otimes \nabla_L ) = \psi(V,\nabla_V) +
\mathrm{Id}_V \otimes \psi(L, \nabla_L) = \psi(V, \nabla_V).$$
\end{proof}

\subsection{Semistability}
We recall the notion of semistability, introduced in \cite{Si}, for
a local system, i.e., a pair $(V, \nabla)$, where $V$ is a vector
bundle over $X$ and $\nabla$ is a connection on $V$.

\begin{defi}
We say that the local system
$(V, \nabla)$ is {\em semistable} (resp. {\em stable}) if any proper $\nabla$-invariant 
subsheaf $W \subset V$ satisfies $\mu(W) \leq \mu(V)$ (resp. $\mu(W) < \mu(V)$).
\end{defi}


In the particular case of a local system $(V, \nabla)$ with $\psi(V, \nabla) = 0$ the
semistability condition can be expressed as follows. In that case,
by Cartier's theorem (Theorem \ref{Cartierthm}), the local system
$(V, \nabla)$ is of the form $(V,\nabla) = (F^*(E), \nabla^{can})$ for some vector bundle $E$ over $X$. Then, since
$\nabla^{can}$-invariant subsheaves of $F^*(E)$ correspond to subsheaves of $E$, the local system $(F^*(E), \nabla^{can})$ is
semistable if and only if $E$ is semistable.

\begin{lem} \label{lemoperss}
Let $(V, \nabla, V_\mydot)$ be an oper of any type with $Q = V_0/V_1$ and let $W
\subset V$ be a $\nabla$-invariant subbundle. We consider the induced filtration
$W_\mydot$ on $W$ defined by $W_i = W \cap V_i$ and denote by $m$ the
integer satisfying $W_m \not= 0$ and $W_{m+1} = 0$. Then
\begin{itemize}
\item[(i)] there is an inclusion $W_0 /W_1 \hookrightarrow Q$ and the connection $\nabla$ induces sheaf inclusions
$$ W_i / W_{i+1} \hookrightarrow  \left( W_{i-1}/ W_i \right) \otimes \Omega_X^1.$$
\item[(ii)] we have a decreasing sequence of integers
\begin{equation}\label{ineqranksW}
\rk(Q) = q \geq r_0\geq r_1\geq \cdots \geq r_m\geq 1 \qquad \text{and} \qquad
\sum_{i = 0}^m r_{i}= \rk(W).
\end{equation}
with $r_i = \rk(W_i / W_{i+1})$ for $0 \leq i \leq m$.
\end{itemize}
\end{lem}

\begin{proof}
Since $W$ is $\nabla$-invariant, the connection $\nabla$ induces the horizontal maps
of the following commutative diagram
$$
\begin{CD}
W_i/W_{i+1} @>>>  \left(W_{i-1}/W_i\right)\otimes \Omega_X^{1} \\
@VVV   @VVV \\
  V_i/V_{i+1} @>>>  \left(V_{i-1}/V_i\right)\otimes \Omega_X^{1}
\end{CD}
$$
Since the vertical arrows are inclusions and the lower horizontal map
is an isomorphism (by Definition \ref{defioper}), we obtain part (i). Part (ii)
follows immediately from part (i).
\end{proof}

\begin{prop} \label{stabilityoper}
Let $(V, \nabla, V_\mydot)$ be an oper of any type with $Q = V_0/V_1$.
If $Q$ is semistable (resp. stable), then the
local system $(V, \nabla)$ is semistable (resp. stable).
\end{prop}

\begin{proof}
We will use Lemma~\ref{lemoperss} and its notational setup. Combining the inclusions of Lemma \ref{lemoperss} (i), we obtain for any $0 \leq i \leq m$
$$ W_i/W_{i+1} \hookrightarrow Q \otimes (\Omega_X^1)^{\otimes i}. $$
Hence, if $Q$ is semistable, we have the inequality $\mu(W_i/W_{i+1}) \leq \mu(Q) +
i(2g-2)$ and we obtain
$$ \deg(W) = \sum_{i=0}^m r_i \mu(W_i/W_{i+1}) \leq \rk(W) \mu(Q) + (2g-2)
\left( \sum_{i=0}^m ir_i \right) $$
or, equivalently,
$$ \mu(W) \leq \mu(Q) + \frac{2g-2}{\rk(W)} \left( \sum_{i=0}^m ir_i \right) $$
By relation  \eqref{degoper} we have $\mu(V) = \mu(Q) + (l-1)(g-1)$, so the
semistability condition $\mu(W) \leq \mu(V)$ will be implied by the inequality
$$ 2 \left( \sum_{i=0}^m ir_i \right) \leq \rk(W) (l-1).$$
Obviously the length $m$ of the induced filtration on $W$ is bounded by $l-1$, hence
it suffices to show that $2 \left( \sum_{i=0}^m ir_i \right) \leq \rk(W)m$. But the
latter inequality reduces to
$$ \rk(W) m -  2 \left( \sum_{i=0}^m ir_i \right) = \sum_{i=0}^m (m-2i) r_i
= \sum_{i=0}^{[\frac{m}{2}]} (m-2i)(r_i - r_{m-i}) \geq 0, $$
which holds because of the inequalities \eqref{ineqranksW}.


Finally, if $Q$ is stable, then one easily deduces from the preceding inequalities that
equality $\mu(W) = \mu(V)$ holds if and only if $W = V$. Hence $(V,\nabla)$ is stable.
\end{proof}


Although we will not use the next result, we mention an interesting corollary (due to X. Sun)

\begin{cor} [\cite{sun06} Theorem 2.2]
Let $E$ be a vector bundle over $X$. If $E$ is semistable (resp. stable), then the
direct image under the Frobenius morphism  $F_*(E)$ is semistable (resp. stable).
\end{cor}

\begin{proof}
It suffices to apply Proposition \ref{stabilityoper} to the oper
$(F^*(F_*(E)), \nabla^{can} , V_\mydot)$, where $V_\mydot$ is the
canonical filtration, which satisfies $V_0/ V_1 = E$ (see Theorem~\ref{exampleoper}).
\end{proof}

\begin{rem}
 We note that Proposition \ref{stabilityoper} and Lemma \ref{lemoperss}
are essentially a reformulation of X. Sun's arguments in the set-up of opers.
\end{rem}


\section{Quot-schemes and Frobenius-destabilized vector bundles}


\subsection{Statement of the results}
Let $r \geq 2$ be an integer and put
$$C(r,g) = r(r-1)(r-2)(g-1).$$
The purpose of this section is to show the following

\begin{thm} \label{QuotFrob}
Let $X$ be a smooth, projective curve of genus $g\geq 2$ over an
algebraically closed field $k$ of characteristic $p>0$. If $p > \crg$, then we have


\begin{enumerate}

\item Every stable and Frobenius-destabilized vector bundle $V$ of rank
$r$ and slope $\mu(V)=\mu$ over $X$ is a subsheaf $V\hookrightarrow F_*(Q)$
for some stable vector bundle $Q$ of rank  $\rk(Q)<r$ and $\mu(Q) < p\mu$.

\item Conversely, given a semistable vector bundle $Q$ with $\rk(Q)<r$ and $\mu(Q) < p\mu$, every
subsheaf $V\hookrightarrow F_*(Q)$ of rank $\rk(V)=r$ and slope
$\mu(V)=\mu$ is semistable and destabilized by Frobenius.
\end{enumerate}

\end{thm}

In the case when the degree of the Frobenius-destabilized bundle equals $0$, we will prove the
following refinement of Theorem \ref{QuotFrob}.

\begin{thm} \label{QuotFrobdegree0}
Let $X$ be a smooth, projective curve of genus $g \geq 2$ over an
algebraically closed field $k$ of characteristic $p$.
If $p > \crg$, then every stable and Frobenius-destabilized vector bundle
$V$ of rank $r$ and of degree $0$ over $X$ is a subsheaf
$$V\hookrightarrow F_*(Q)$$
for some stable vector bundle $Q$ of rank $\rk(Q)<r$ and degree $\deg(Q)=-1$.
\end{thm}

\subsection{Proof of Theorem \ref{QuotFrob}}

Let $V$ be a stable and Frobenius-destabilized vector bundle of rank $r$ and slope $\mu(V)=\mu$.
Consider the first quotient $Q$ of the Harder-Narasimhan filtration of $F^*(V)$. If $Q$ is not
stable, we replace $Q$ by a stable quotient. This shows the existence of a stable
vector bundle $Q$
such that
$$F^*(V)\to Q \qquad \text{and} \qquad p\mu=\mu(F^*(V)) > \mu(Q).$$
Moreover $\rk(Q) < \rk(V)=r$. By
adjunction we obtain a non-zero map
$$V\to F_*(Q).$$
Thus to prove the Theorem \ref{QuotFrob}(1) it will suffice to
prove that this map will be an injection.

Suppose that this is not the case. Then the image of $V\to F_*(Q)$
generates a subbundle, say, $W\subset F_*(Q)$ and one has $1\leq
\rk(W)\leq r-1$ and by the stability of $V$, we have
$$\mu(V)=\mu<\mu(W).$$
Now we observe that we can bound $\mu(W)$ from below
$$ \mu(W) \geq \mu  + \frac{1}{r(r-1)} > \frac{\mu(Q)}{p} +  \frac{1}{r(r-1)}.$$
The proof of Theorem \ref{QuotFrob}(1)
will now follow from Proposition \ref{slopesubbundlesFQ} below applied with $\delta = \frac{1}{r(r-1)}$ and
$n= r-1$.


Let us also indicate how to deduce Theorem \ref{QuotFrob}(2)
from Proposition \ref{slopesubbundlesFQ} as well. Let $V\subset
F_*(Q)$ be a subsheaf of $\rk(V)=r$ and $\mu(V)=\mu$ with
$\mu(Q)/p \leq \mu$. If $V$ is not semi-stable, then there exists a
subsheaf $W\subset V$ with $\mu(W)>\mu(V)$ and,  more precisely,
$\mu(W) \geq \mu(V) + \frac{1}{r(r-1)}$. Moreover $W\subset
V\subset F_*(Q)$.  Replacing $W$ by a subbundle generated by it in
$F_*(Q)$, we see that we have a subbundle $W \subset F_*(Q)$ with
$\mu(W) \geq \frac{\mu(Q)}{p}  + \frac{1}{r(r-1)}$. We can now apply
Proposition \ref{slopesubbundlesFQ} with $\delta = \frac{1}{r(r-1)}$ and
$n= r-1$ to obtain a contradiction. Hence $V$ is semistable. The fact that
$V$ is Frobenius-destabilized follows immediately by adjunction. This completes
the proof of Theorem \ref{QuotFrob}.


\begin{prop} \label{slopesubbundlesFQ}
Let $Q$ be a semistable vector bundle over the curve $X$. Let $\delta>0$ be a real number and let $n$ be a positive
integer. Assume that $p > \frac{(n-1)(g-1)}{\delta}$. Then any subbundle $W \subset
F_*(Q)$ of rank $\rk(W) \leq n$ has slope
$$\mu(W) < \frac{\mu(Q)}{p} + \delta.$$
\end{prop}

\begin{proof}
We recall from Theorem~\ref{exampleoper} that the triple
$(F^*(F_*(Q)),\nabla^{can},V_\mydot)$ is a dormant oper
of type $q = \rk(Q)$. Let $W \subset F_*(Q)$ be a subbundle. Then let $W_0=F^*(W) \subset
F^*(F_*(Q))= V$ and consider the induced flag on $W_0$, i.e.
$W_i=W_0\cap V_i$. We now apply Lemma \ref{lemoperss} (i) to the subbundle $W_0 \subset
V$ and, using the notation of that lemma, we have
\begin{equation}\label{constraint1}
q\geq r_0\geq r_1\geq \cdots \geq r_m\geq 1 \qquad \text{and} \qquad
\sum_{i = 0}^m r_{i}= \rk(W)=w.
\end{equation}
It is shown in \cite{sun06} formula (2.12) that
\begin{equation} \label{formulasun}
\mu(F_*(Q))\geq
\mu(W)+\frac{2(g-1)}{pw}\sum_{i=0}^m\left(\frac{p-1}{2}-i \right)r_i.
\end{equation}
We will bound the sum on the right from below to obtain a bound on
$\mu(F_*(Q))-\mu(W)$. This will be a substantial strengthening of the
results of \cite{JRXY}, \cite{lange-pauly08}  and \cite{joshi04}.

We have the equalities
\begin{eqnarray*}
\sum_{i=0}^m\left(\frac{p-1}{2}-i\right)r_i&=&\sum_{i=0}^m\frac{p-1}{2}r_i-\sum_{i=0}^m i
r_i\\
&=&\frac{p-1}{2}\sum_{i=0}^mr_i-\sum_{i=0}^m i r_i,\\
&=&w\frac{p-1}{2}-\sum_{i=0}^m i r_i.
\end{eqnarray*}
So to bound the expression on the right from below, it will suffice
to bound the sum $\sum_{i=0}^m i r_i$ from above subject
to constraints \eqref{constraint1}. This is done in the following

\begin{lem}
For any sequence of positive integers
$r_\mydot = (r_0, r_1, \ldots, r_m)$ satisfying the constraints \eqref{constraint1}
we have the inequality
$$ S(r_\mydot) = \sum_{i=0}^{m} i r_{i} \leq \frac{w(w-1)}{2}.$$
Moreover, we have equality if and only if $m= w-1$ and $r_0 = \cdots = r_m = 1$.
\end{lem}

\begin{proof}
Given a sequence $r_\bullet$ we introduce the sequence of integers $s_\mydot = (s_0, \ldots , s_{m+1})$ with $s_i \geq 0$
defined by the relations
$$ r_i = 1 + s_{m+1} + s_m + \ldots + s_{i +1} \qquad \text{for} \ 0 \leq i \leq m, \qquad \text{and}
\qquad r_0 + s_0 = q.$$
We compute $S(r_\mydot) = \frac{m(m+1)}{2} + \sum_{i = 1}^{m+1} \frac{i(i - 1)}{2} s_i$.
The problem
is then equivalent to determine an upper bound for $S(r_\mydot)$, where the integers $s_\mydot$ are subject to
the two constraints $\sum_{i = 0}^{m+1} s_i = q-1$ and $\sum_{i = 1}^{m+1} i s_i = w - (m+1)$.
Since $s_i \geq 0$ and $\frac{i(i -1)}{2} \leq \frac{i m}{2}$ for $1 \leq i \leq m+1$, we obtain
the inequality
$$ \sum_{i =1}^{m+1} \frac{i(i -1)}{2} s_i  \leq \sum_{i = 1}^{m+1} \frac{m}{2} i s_i =
\frac{m}{2} ( w - (m+1)),$$
hence $S(r_\bullet) \leq \frac{mw}{2}$. If we vary $m$ in the range $0 \leq m \leq w-1$, we observe that the
maximum is obtained for $m= w -1$ and that equality holds for $s_1 = \ldots = s_{m+1} = 0, \ s_0 = q-1$, i.e.
$r_0 = \cdots = r_m = 1$.
\end{proof}

Combining the previous lemma with inequality \eqref{formulasun}, we obtain
$$ \mu(F_*(Q)) \geq
\mu(W)+\frac{2(g-1)}{pw}\left[w\frac{p-1}{2}-\frac{w(w-1)}{2}\right] = \mu(W)+(g-1)\left(1-\frac{w}{p}\right).
$$
As by Lemma \ref{degFQ} we have $\mu(F_*(Q)) = \frac{\mu(Q)}{p} + (g-1)(1 - \frac{1}{p})$, we obtain the inequality
$$ \mu(W) \leq \frac{\mu(Q)}{p} + (g-1)\left( \frac{w-1}{p}  \right).$$
Hence if $\frac{(g-1)(w-1)}{p} < \delta$, or equivalently $p > \frac{(w-1)(g-1)}{\delta}$, we
obtain the desired inequality.
\end{proof}

\subsection{Proof of Theorem \ref{QuotFrobdegree0}}

Let $V$ be a Frobenius-destabilized vector bundle of degree $0$ and rank $r$. We will need a lemma.

\begin{lem}
Let $E$ be a non-semistable vector bundle with $\deg E = 0$. Then there exists a semistable
vector bundle $Q$ with $\deg(Q) = -1$ and $\rk(Q) < \rk(E)$ and such that $\Hom(E, Q) \not= 0$.
\end{lem}

\begin{proof}
It suffices to prove the existence of a semistable subsheaf $S \subset V= E^*$ with $\deg S = 1$ and
then take the dual. We will prove that by a double induction. Given two positive integers $r$ and $d$
we introduce the induction hypothesis $H(r,d)$: ``If $V$ contains a semistable subsheaf of rank $r$ and
degree $d$, then $V$ contains a semistable subsheaf of degree $1$".

We will show that $H(r,d)$ holds for any $1 \leq r < \rk(E)$ and any $d \geq 1$.  We first observe that $H(r,1)$ holds
trivially for any $r$, and $H(1,d)$ holds for any $d$, because there exist degree $1$ and rank $1$ subsheaves of any degree
$d$ line bundle.


We now suppose that $H(r,d)$ holds for any pair $(r,d)$ with $1 \leq r < r_0$ and we will deduce that $H(r_0,d)$
holds for any $d \geq 1$. As observed above $H(r_0,1)$ holds. So assume that $H(r_0, d)$ holds. We will
show that $H(r_0, d+1)$ also holds. So assume that $V$ contains a semistable subsheaf $W$ of rank $r$ and
degree $d+1$. Two cases can occur\footnote{We are grateful to Peter Newstead for having pointed out this fact.}:
either there exists a semistable negative elementary transform $W' \subset W$ of colength one (hence a general
negative elementary transform is semistable by openness of semistability) or any such $W'$
is non-semistable.  In the first case we can take a semistable elementary transform $W'$, which has degree $d$,
and apply $H(r_0,d)$, which holds by induction. In the second case we pick any elementary transform $W'$ and
consider the first piece $M \subset W'$ of its Harder-Narasimhan filtration. Then it follows from the definition
of $M$ that
$$ \frac{d}{r_0} = \mu(W') < \mu(M) \qquad \text{and} \qquad \rk(M) < r_0.$$
In particular, $\deg(M) \geq 1$. So we can apply $H(\rk(M), \deg(M))$, which holds by induction, and we are done.
\end{proof}


We apply the lemma to the non-semistable bundle $V = F^*(E)$ and obtain a non-zero map
$$ F^*(E) \rightarrow Q \qquad \text{with} \qquad \deg(Q) = -1, \ \rk(Q) \leq r-1.$$
We now continue as in the proof of Theorem \ref{QuotFrob}(1) to show that the map $E \rightarrow
F_*(Q)$ obtained by adjunction is injective.


\subsection{Loci of Frobenius-destabilized vector bundles} \label{loci}

Let $\cM(r)$ denote the coarse moduli space of $S$-equivalence classes of semistable bundles
of rank $r$ and degree $0$ over the curve $X$.

We wish to define the closed subscheme of $\cM(r)$ which parameterizes Frobenius-destabilized bundles. The construction is as follows. We recall the construction of $\cM(r)$ as a projective variety via Geometric Invariant Theory. There exists a quasi-projective variety $R$ and an action of 
$PGL(N)$ on $R$ such that the quotient is $\cM(r)$. On $R$, we have a universal family of semistable bundles on $X$. So on $R$ we may apply
a theorem of S. Shatz (see \ref{thmshatz} or \cite{shatz}) which asserts that the locus of points of $R$ where 
the Harder-Narasimhan polynomial jumps is closed in $R$. We will abuse notation by identifying a 
semistable bundle $V$ on $X$ with the corresponding point $[V]\in R$. Let $Z\subset R$ be the locus of 
points $V\in R$, where $F^*(V)$ has Harder-Narasimhan polygon lying strictly above the slope zero polygon of rank $r$. 
Then $Z$ is closed by Shatz' Theorem and invariant under the action of $PGL(N)$, since,
given a strictly semistable bundle $E$ with associated graded $\mathrm{gr}(E) =
E_1 \oplus \cdots \oplus E_l$ with $E_i$ stable, one observes that $E$ is
Frobenius-destabilized if and only if at least one of the stable summands $E_i$ is
Frobenius-destabilized. Then let $\cJ(r)$ be the image of $Z$ under the quotient morphism $R\to \cM(r)$. As $\cM(r)$ is a good quotient, the image of $Z$ is closed. So $\cJ(r)\subset\cM(r)$ is closed.
Moreover, $\cJ(r)$ is a closed subvariety of $\cM(r)$. Let $\cJ^s(r) \subset \cJ(r)$ be the open subset corresponding to stable bundles. This is the image of under the above quotient map, of the corresponding subset $Z^s\subset Z$ with $V$ stable.
We will refer to $$ \cJ(r) \subset \cM(r) $$
as the locus of semistable bundles $E$ which are destabilized by Frobenius pull-back, i.e.
$F^*(E)$ is not semistable.


Let $1 \leq q \leq r-1$ be an integer and let $\cM(q,-1)$ be the moduli
space of semistable bundles of rank $q$ and degree $-1$ over the curve $X$. As $\mathrm{gcd}(q,-1) = 1$ we are in the
coprime case and so every semistable bundle $Q \in \cM(q,-1)$ is stable. In
particular, we see that there exists a universal Poincar\'e bundle $\cU$
on $\cM(q,-1) \times X$. Let
$$ \alpha: \QQuot(q,r,0) := \mathrm{Quot}^{r,0}((F \times \mathrm{id}_{\cM(q,-1)})_* \cU)
\lra \cM(q,-1) $$
be the relative Quot-scheme \cite{Gro} over $\cM(q,-1)$. The fibre $\alpha^{-1}(Q)$
over a point $Q \in \cM(q,-1)$ equals $\mathrm{Quot}^{r,0}(F_*(Q))$. We note that
$\QQuot(q,r,0)$ is a proper scheme \cite{Gro}.


We are now ready to restate Theorems \ref{QuotFrobdegree0} and \ref{QuotFrob}(2)  in a geometrical set-up.

\begin{thm} \label{locusJ}
If $p > \crg$, then the image of the forgetful morphism
$$ \pi : \coprod_{q=1}^{r-1} \QQuot(q,r,0) \lra \cM(r), \qquad [E \subset F_*(Q)] \mapsto E $$
is contained in the locus $\cJ(r)$ and contains the closure $\overline{\cJ^s(r)}$
of the stable locus $\cJ^s(r)$.
\end{thm}

\begin{ques}
For $r\geq 3$, is the locus $\cJ(r)$ equal to the closure $\overline{\cJ^s(r)}$
in $\cM(r)$? In other words, does any irreducible component of $\cJ(r)$ contain
stable bundles?
\end{ques}


\subsection{Maximal degree of subbundles of $F_*(Q)$}

The next proposition will not be used in this paper.

\begin{prop} \label{maxdegree}
Let $Q$ be a semistable vector bundle of rank $q = \rk{Q}$ satisfying
$$ q < r < pq \qquad \text{and} \qquad -(r-q)(g-1) \leq \deg(Q) < 0. $$
If $p > r(r-1)(g-1)$, then the maximal degree of rank-$r$ subbundles of $F_*(Q)$ equals $0$.
In particular, any subsheaf of $F_*(Q)$ of degree $0$ is a subbundle.
\end{prop}

\begin{proof}
By Proposition \ref{nonemptyQuot} we know that $\mathrm{Quot}^{r,0}(F_*(Q)) \not= \emptyset$, hence
the maximal degree of rank-$r$ subbundles is at least $0$. It is therefore enough to show
that any rank-$r$ subbundle $W \subset F_*(Q)$ satisfies $\mu(W) \leq 0$. We apply
Proposition \ref{slopesubbundlesFQ} with $n = r$ and $\delta = \frac{1}{r}$, which leads to
$\mu(W) < \frac{1}{r}$, hence $\mu(W) \leq 0$.
\end{proof}


\section{Harder-Narasimhan polygons of local systems}


\subsection{Harder-Narasimhan filtration}\label{HNdef}

Given a vector bundle $V$ over $X$ we consider its Harder-Narasimhan filtration
$$ V_\bullet^{HN} : \qquad  0 = V_l \varsubsetneq V_{l-1}   \varsubsetneq \cdots \varsubsetneq V_1 \varsubsetneq V_0 = V$$
and we denote for $1 \leq i \leq l$ the slopes of the successive quotients $\mu_i = \mu(V_{i-1}/ V_{i})$, which
satisfy
$$ \mu_{l} > \mu_{l-1} > \cdots > \mu_2 > \mu_1.$$
We will also use the notation
$$ \mu_{max}(V) = \mu_l \qquad \text{and} \qquad \mu_{min}(V) = \mu_1.$$
Then the numerical information $(\rk(V_i),\deg(V_i))$
associated to the Harder-Narasimhan flag can be conveniently
organized into a convex polygon in the plane $\RR^2$, denoted by $\sP_{V}$. It is the
convex polygon with vertices (or ``break points") at the
points $(\rk(V_i),\deg(V_i))$ for $0 \leq i \leq l$. The segment joining
$(\rk(V_{i}),\deg(V_{i}))$ and $(\rk(V_{i-1}),\deg(V_{i-1}))$ has slope
$\mu_{i}$.


We will use the following result (see \cite{laszlo-pauly} Lemma 4.2, \cite{sun99} Theorem 3.1 or \cite{shepherd} Corollary 2)

\begin{lem} \label{inequalityslopelocalsystem}
If the local system $(V, \nabla)$ is semistable, then we have for $1 \leq i \leq l-1$
$$ \mu_{i+1} - \mu_i \leq 2g-2.$$
\end{lem}

\begin{rem}
Note that \cite{sun99} and \cite{shepherd} make the assumption that $\psi(V, \nabla) = 0$.
In fact, this assumption is not needed in their proofs.
\end{rem}

\subsection{A theorem of Shatz}

The set of convex polygons in the plane $\RR^2$ starting at $(0,0)$ and ending at $(r,0)$ is partially ordered:
if $\sP_1,\sP_2$ are two such polygons, we say that $\sP_1 \succcurlyeq \sP_2$ if $\sP_1$ lies on or above $\sP_2$. We recall the
following well-known theorem by Shatz \cite{shatz}

\begin{thm} \label{thmshatz}
\begin{enumerate}
\item
Let $V$ be a family of vector bundles on $X\times T$ parameterized
by a scheme $T$ of finite type. For $t \in T$, we denote by $V_t =
V|_{X\times_Tk(t)}$ the restriction of $V$ to the fibre over $t$.
For any convex polygon $\sP$ the set
$$ S_{\sP} = \{ t \in T \ | \  \sP_{V_t}  \succcurlyeq  \sP \} $$
is closed in $T$ and the union of the subsets $S_\sP$ gives a stratification of $T$.
\item
Let $R$ be a discrete valuation ring and let $V$ be a family of
vector bundles parameterized by $\Spec(R)$, then
$$ \sP_{V_s} \succcurlyeq \sP_{V_\eta},$$
where $s$ and $\eta$ denote the closed and generic point of $\Spec(R)$ respectively.
\end{enumerate}
\end{thm}

\subsection{Oper-polygons are maximal}

Let $(V, \nabla, V_\bullet)$ be an oper of rank $r$, degree $0$ and type $1$. Then it follows
from \eqref{quotientoper} and \eqref{degoper} that $l = r$ and for $0 \leq i \leq r$
$$ \rk(V_i) = r-i, \qquad \deg(V_i) = i(r-i)(g-1).$$
We introduce the oper-polygon
$$\sP_r^{oper} : \qquad  \text{with vertices} \ \ (i,i(r-i)(g-1))
\ \text{for}  \ 0 \leq i \leq r.$$
In the next section we will see that the oper-polygon actually appears as
a Harder-Narasimhan polygon of some semistable local system. Our next result says that
oper-polygons are maximal among Harder-Narasimhan polygons of semistable local systems.

\begin{thm} \label{thmoperdominant}
Let $(V, \nabla)$ be a semistable local system of rank $r$ and degree $0$. Then
\begin{enumerate}
\item We have the inequality
$$ \sP_r^{oper} \succcurlyeq \sP_V. $$
\item The equality
$$ \sP_r^{oper} = \sP_V $$
holds if and only if the triple $(V, \nabla, V_\bullet^{HN})$ is an oper.
\end{enumerate}
\end{thm}

\begin{proof}
Let $(V, \nabla)$ be a semistable local system of rank $r$ and degree $0$. Assume that $V$ is
not semistable. Let
$$ 0 = V_l \varsubsetneq V_{l-1}   \varsubsetneq \cdots \varsubsetneq V_1 \varsubsetneq V_0 = V$$
be the Harder-Narasimhan filtration of $V$. We denote $n_i = \rk (V_{i-1}/ V_i)$ for
$1 \leq i \leq l$, so that
$$ n_1 + n_2 + \cdots + n_l = r \qquad \text{and} \qquad \rk(V_i) = n_l + n_{l-1} + \cdots + n_{i+1}.$$
So in order to show the first part of the theorem, we have to prove that
\begin{equation} \label{ineqdeg}
 \deg(V_i) \leq (g-1) (n_1 + \cdots + n_i)(n_{i+1} + \cdots + n_l)
\end{equation}
holds for every $1 \leq i \leq l-1$. We also denote $\delta_i = \deg(V_{i-1}/ V_i)$, so that
$\mu_i = \mu(V_{i-1}/ V_i) = \frac{\delta_i}{n_i},$
$$ \delta_1 + \delta_2 + \cdots + \delta_l = 0 \qquad \text{and} \qquad
\deg(V_i) = \delta_l + \delta_{l-1} + \cdots + \delta_{i+1}.$$
The inequalities of Lemma \ref{inequalityslopelocalsystem} give for $1 \leq i \leq l-1$
$$
(\cE_i): \qquad \frac{\delta_{i+1}}{n_{i+1}} \leq \frac{\delta_i}{n_i} + 2g-2.
$$
We will prove \eqref{ineqdeg} by a decreasing induction on $i$. The first step is to establish that
\begin{equation} \label{degVl}
\delta_{l} = \deg(V_{l-1}) \leq (g-1) (n_1 + \cdots + n_{l-1})n_l.
\end{equation}
We consider for $k= l-1, \ldots, 1$ the inequality $(\cE_{l-1}) + (\cE_{l-2}) + \cdots + (\cE_k)$ multiplied by $n_{k}$:
\begin{eqnarray*}
  \frac{n_{l-1}}{n_l}\delta_l & \leq & \delta_{l-1} + n_{l-1}(2g-2), \\
  \frac{n_{l-2}}{n_l}\delta_l & \leq & \delta_{l-2} + 2n_{l-2}(2g-2), \\
  \vdots & \leq &  \vdots \\
  \frac{n_{1}}{n_l}\delta_l & \leq & \delta_{1}+ (l-1)n_{1}(2g-2).
\end{eqnarray*}
Now using $\sum_{i=1}^l\delta_i=0$ and adding up all these equations and
simplifying we get:
\begin{eqnarray*}
  \frac{n_1+\cdots+n_{l-1}}{n_l}\delta_l & \leq & -\delta_l +2(n_{l-1} +2 n_{l-2} +\cdots+(l-1)n_1)(g-1),\\
  \delta_l & \leq & n_l(g-1)(\frac{2}{r}(n_{l-1}+ \cdots+(l-1)n_1).
\end{eqnarray*}
Thus it remains to show that
$$2(n_{l-1} +2 n_{l-2} +\cdots+(l-1)n_1)\leq r(n_1+ n_2 + \cdots + n_{l-1}).$$
We introduce $m_i = n_i - 1 \geq 0$, and the previous inequality can be
written as follows:
$$2(m_{l-1} +2 m_{l-2} +\cdots+(l-1)m_1+(1+2+\cdots+(l-1)))\leq(l+m_1+\cdots+m_l)((l-1)+m_1+\cdots+m_{l-1}),$$
and after some simplification using $1+2+\cdots+(l-1)=\frac{l(l-1)}{2}$ we
get
\begin{eqnarray*}
2(m_{l-1} +2 m_{l-2} +\cdots+(l-1)m_1)& \leq &
l(m_1+\cdots+m_{l-1})+(l-1)(m_1+\cdots+m_l) \\
 &  & +(m_1+\cdots+m_{l-1})(m_1+\cdots+m_l).
\end{eqnarray*}
After a rearrangement we are reduced to proving
$$2(m_{l-1} +2 m_{l-2} +\cdots+(l-1)m_1)\leq
(2l-1)(m_1+\cdots+m_{l-1})+\text{non-negative terms}.$$
Now clearly this
holds as the inequality
$$2(m_{l-1} +2 m_{l-2} +\cdots+(l-1)m_1)\leq
(2l-1)(m_1+\cdots+m_{l-1})$$ visibly holds for all $l\geq 2$ as
$m_i\geq0$ and each term on the left is less than or equal to the
corresponding term on the right. This establishes \eqref{degVl}.


Now we will proceed by induction on $i$. So assume that if we have, for some $i$,
the inequality
\begin{equation} \label{ineqdegVi}
\deg(V_i) = \delta_l + \cdots + \delta_{i+1} \leq (g-1)(n_1+\cdots+n_i)(n_{i+1}+\cdots+n_l).
\end{equation}
Then we claim that we also have the corresponding inequality for
$i-1$. The new inequality which needs to be established is
\begin{equation} \label{ineqdegVi-1}
\deg(V_{i-1}) = \delta_l + \cdots + \delta_{i+1} + \delta_i \leq (g-1)(n_1+\cdots+n_{i-1})(n_{i}+\cdots+n_l).
\end{equation}
We begin with the following inequalities:
\begin{eqnarray*}
  \delta_{i} & = &\delta_{i}  \\
  \frac{n_{i-1}}{n_{i}}\delta_{i} &\leq & \delta_{i-1} + n_{i-1}(2g-2) \\
  \frac{n_{i-2}}{n_{i}}\delta_{i} &\leq & \delta_{i-2} + 2 n_{i-2}(2g-2)  \\
  \vdots & \leq & \vdots\\
  \frac{n_{1}}{n_{i}}\delta_{i}&\leq&\delta_{1}+(i-1)n_{1}(2g-2).
\end{eqnarray*}
obtained by doing the following operations
$n_{k}((\cE_{i-1})+ \cdots + (\cE_{k}))$ for $k = i-1, \ldots, 1$. Now we add all the
inequalities for $k = i-1, \ldots, 1$ and the equality
$\delta_{i}=\delta_{i}$ and we obtain:
\begin{equation*}
    \frac{n_{1} +\cdots+n_i}{n_{i}}\delta_{i}\leq
    (\delta_{1}+ \cdots + \delta_i)+(g-1)2(n_{i-1}+2n_{i-2}+\cdots+(i-1)n_1).
\end{equation*}
Multiplying this by $\frac{n_{i}}{n_{1} + \cdots + n_{i}}$ and using
$\delta_1 + \cdots + \delta_i = -(\delta_{i+1} + \cdots + \delta_l)$ we obtain
$$\delta_{i} + \frac{n_{i}}{n_{1}+\cdots+n_i}(\delta_{i+1} + \cdots + \delta_{l}) \leq
(g-1)2\frac{n_{i}(n_{i-1}+2n_{i-2}+ \cdots + (i-1)n_1)}{n_{1}+ \cdots + n_i}.$$
Next we multiply the inequality \eqref{ineqdegVi}  by
$\frac{n_{1} + \cdots + n_{i-1}}{n_{1}+\cdots+n_i}$ and add to the
previous inequality. We get
\begin{eqnarray*}
    \delta_i + \delta_{i+1} + \cdots + \delta_l & \leq &
    (g-1)(n_1+\cdots+n_{i-1})(n_{i+1}+\cdots+n_l) \\
    & & +(g-1)2n_{i}\frac{n_{i-1}+2n_{i-2}+\cdots+(i-1)n_i}{n_{1}+\cdots+n_i}.
\end{eqnarray*}
So in order to show that the required inequality \eqref{ineqdegVi-1} holds, it is enough to establish that
\begin{equation*}
2\left(\frac{n_{i-1}+2n_{i-2}+\cdots+(i-1)n_1}{n_{1}+\cdots+n_i} \right) \leq n_{1}+ \cdots + n_{i-1}.
\end{equation*}
Equivalently we have to establish that
$$2(n_{i-1}+2n_{i-2}+\cdots+(i-1)n_1)\leq (n_{1}+\cdots+n_{i-1})(n_{1}+\cdots+n_i).$$
But this follows immediately by introducing $m_i=n_i-1\geq 0$ as was done in the first step.
Since this is straightforward, we omit the details. This finishes the proof of the first part.


In order to show part two, we will directly check that a triple $(V, \nabla, V_\bullet^{HN})$ with
$\sP_V = \sP_r^{oper}$ satisfies the two conditions of Definition \ref{defioper}. Given an
integer $i$ with $1 \leq i \leq l-1$ we consider the $\cO_X$-linear map $\nabla_i$ induced by the
connection $\nabla$
$$ \nabla_i : V_i \longrightarrow (V/V_i) \otimes \Omega^1_X.$$
We consider the composite $\overline{\nabla}_i$ of $\nabla_i$ with the canonical projection $(V/V_i) \otimes \Omega^1_X \rightarrow
(V/V_{i-1}) \otimes \Omega^1_X$. On one hand
$$ \mu_{max}((V/V_{i-1}) \otimes \Omega^1_X) = \mu((V_{i-2}/V_{i-1}) \otimes \Omega_X^1) = \mu(V_{i-2}/V_{i-1}) + 2g - 2, $$
and on the other hand
$$ \mu_{min}(V_i) = \mu(V_i/V_{i+1}) = \mu(V_{i-2}/V_{i-1}) + 4g-4 > \mu_{max}((V/V_{i-1}) \otimes \Omega^1_X).$$
The last inequality implies that $\overline{\nabla}_i = 0$, hence $\nabla(V_i) \subset V_{i-1} \otimes \Omega^1_X$
for $1 \leq i \leq l-1$.


Since the local system $(V, \nabla)$ is semistable, the map $\nabla_i$ is nonzero. Moreover,
since $\nabla(V_{i+1}) \subset V_{i} \otimes \Omega^1_X$, the map $\nabla_i$ factorizes
as follows
$$ \tilde{\nabla}_i :  V_i/ V_{i+1} \longrightarrow (V_{i-1}/ V_i) \otimes \Omega^1_X.$$
Since both sides are line bundles of the same degree, we conclude that $\tilde{\nabla}_i$
is an isomorphism.
\end{proof}


\subsection{Correspondence between dormant opers and Quot-schemes} \label{sectionFroboperQuot}

In this section we will determine all semistable local systems $(V, \nabla)$ with
$\sP_V = \sP_r^{oper}$ and $\psi(V,\nabla) = 0$.

\begin{thm} \label{opersquot}
Let $r \geq 2$ be an integer and assume $p > \crg$. Then we have
\begin{enumerate}
\item Given a line bundle $Q$ of degree $\deg(Q) = -(r-1)(g-1)$, the Quot-scheme $\mathrm{Quot}^{r,0}(F_*(Q))$ is non-empty and any vector bundle
$W \in \mathrm{Quot}^{r,0}(F_*(Q))$
gives under pull-back by the Frobenius morphism a semistable local system
$$ (F^*(W), \nabla^{can}) \qquad \text{with} \qquad \sP_{F^*W} = \sP_r^{oper},$$
i.e., the triple $(F^*(W), \nabla^{can}, (F^*(W))^{HN}_\bullet)$ is a dormant oper.
\item Conversely, any dormant oper of degree $0$ is of the form
$(F^*(W), \nabla^{can}, (F^*(W))^{HN}_\bullet)$ with $W \in \mathrm{Quot}^{r,0}(F_*(Q))$ for
some line bundle $Q$ of degree $\deg(Q) = -(r-1)(g-1)$.
\end{enumerate}
\end{thm}

\begin{proof}
First we note that the non-emptiness of
$\mathrm{Quot}^{r,0}(F_*(Q))$ has been shown in Proposition
\ref{nonemptyQuot}. Secondly, by Theorem \ref{QuotFrob} (2) any
vector bundle $W \in \mathrm{Quot}^{r,0}(F_*(Q))$ is semistable,
hence $(F^* W, \nabla^{can})$ is a semistable local system. Thus it
remains to check that $\sP_{F^*W} = \sP_r^{oper}$. As in the proof
of Proposition \ref{slopesubbundlesFQ} we induce the oper flag
$V_\bullet$ of $V = F^*(F_*(Q))$ on $W_0 = F^*(W)$. Using the same
notation as in the proof of Proposition \ref{slopesubbundlesFQ}, we
immediately deduce from the inequalities \eqref{constraint1} that
$m= r-1$ and $r_0 = \cdots = r_m = 1$. Moreover, since $W_i/
W_{i+1}$ is a subsheaf of $V_i/V_{i+1}$, we have the inequalities
for $0 \leq i \leq m$
$$ \deg(W_i/ W_{i+1}) \leq \deg( V_i/V_{i+1}) = \deg(Q) + i(2g-2).$$
Summing over $i$ we obtain
$$ 0 = \deg(W_0) = \sum_{i=0}^m \deg(W_i/ W_{i+1}) \leq \sum_{i=0}^m \deg( V_i/V_{i+1}) = (m+1) \deg(Q) + \frac{m(m+1)}{2}(2g-2) = 0.$$
Hence we deduce that the previous inequalities are equalities
$\deg(W_i/ W_{i+1}) = \deg(Q) + i(2g-2)$, so that the induced flag $W_\bullet$ on $F^*(W)$ is in fact the oper flag. Moreover, since the quotients $W_i/W_{i+1}$ are line bundles and the sequence of slopes
$\mu(W_i/W_{i+1})$ is strictly increasing, the flag $W_\bullet$ coincides with the Harder-Narasimhan filtration of
$F^*(W)$. This proves part one.


In order to show part two, we observe that any dormant oper is of the form $(F^*(W), \nabla^{can}, V_\bullet)$, where by Remark \ref{flagHN} the
oper flag $V_\bullet$ is necessarily the Harder-Narasimhan filtration $(F^*(W))^{HN}_\bullet$. If we denote by $Q$ the
line bundle quotient $V_0/ V_1$, then $\deg(Q) = -(r-1)(g-1)$, and by adjunction we obtain a non-zero map
$W \rightarrow F_*(Q)$. We then conclude, as in the proof of Proposition \ref{QuotFrob}, that $W \hookrightarrow
F_*(Q)$ is an injection.
\end{proof}


The previous result leads to a description of the set of all dormant opers of
degree $0$ as a relative Quot-scheme. We need to introduce some notation. We consider a
universal line bundle $\cU$ over $X \times \Pic^{-(r-1)(g-1)}(X)$ and
denote by
$$ \alpha: \QQuot(r,0) := \mathrm{Quot}^{r,0}((F \times \mathrm{id}_{\Pic})_* \cU)
\lra \Pic^{-(r-1)(g-1)}(X) $$
the relative Quot-scheme over the Picard variety $\Pic^{-(r-1)(g-1)}(X)$.
The fibre $\alpha^{-1}(Q)$ over a line bundle $Q \in \Pic^{-(r-1)(g-1)}(X)$ equals the
Quot-scheme $\mathrm{Quot}^{r,0}(F_*(Q))$.


The group $\Pic^0(X)$ parameterizing degree $0$ line bundles over
$X$ acts on the relative Quot-scheme $\QQuot(r,0)$ via tensor
product. Note that by the projection formula
$$\text{if} \  W \in \mathrm{Quot}^{r,0}(F_*(Q)), \qquad \text{then} \
W \otimes L  \in \mathrm{Quot}^{r,0}(F_*(Q \otimes L^{\otimes p})).$$
Using this action one observes that the fibres $\alpha^{-1}(Q)$ are all isomorphic
for $Q \in \Pic^{-(r-1)(g-1)}(X)$. On the other hand the determinant map
$$ \det : \QQuot(r,0) \lra \Pic^0(X), \qquad W \mapsto \det W $$
gives another fibration of $\QQuot(r,0)$ over $\Pic^0(X)$ and,
using again the action of $\Pic^0(X)$, one notices that the fibres of $\det$
are all isomorphic. We denote by
$$\QQuot(r,\cO_X) := \mathrm{det}^{-1}(\cO_X)$$
the fibre over $\cO_X$.


Then Theorem \ref{opersquot} implies the following

\begin{prop} \label{bijFrobopersQuot}
If $p > \crg$, there is a one-to one correspondence between the
set of dormant opers of rank $r$ and degree $0$ (resp. with fixed
trivial determinant)
and the relative Quot-scheme $\QQuot(r,0)$ (resp. $\QQuot(r,\cO_X)$).
\end{prop}

\begin{rem}
By Proposition \ref{stabilityoper} any oper (of type $1$) is stable, which implies
that, if $p > \crg$, any vector bundle $W \in \mathrm{Quot}^{r,0}(F_*(Q))$ is stable.
\end{rem}





\section{Finiteness of the scheme of nilpotent $\PGL(r)$-opers}


The purpose of this section is to show that there are only a finite number
of dormant opers with trivial determinant (Corollary \ref{finiteFroboper}).
This finiteness result
will imply that certain Quot-schemes have the expected dimension (Theorem \ref{Quotexpdim}).


\subsection{Nilpotent opers}

\begin{defi}
We say that an oper $(V, \nabla, V_\mydot)$ is nilpotent if its $p$-curvature $\psi(V,\nabla)$
is nilpotent.
\end{defi}

We remark that by Proposition \ref{proppcurv} (ii) the property of being nilpotent is
{\em not} invariant under tensor product by rank-$1$ local systems. For $\PGL(r)$-opers
we therefore take the following

\begin{defi}
We say that a $\omega \in \Op_{\PGL(r)}$ is nilpotent  (resp. dormant)
 if some (hence any) lift
$\sigma_\theta (\omega) \in
\Op_{\SL(r)}$ is nilpotent (resp. dormant).
\end{defi}

By Proposition  \ref{HMI} we obtain that $\omega$ is nilpotent if and only if
$\HM(\omega) = 0$. We
will denote by $\Nilp_r(X) := \HM^{-1}(0) \subset \Op_{\PGL(r)}$
the fiber over $0$ of the Hitchin-Mochizuki map. It parameterizes
nilpotent $\PGL(r)$-opers and contains in particular dormant
$\PGL(r)$-opers.

\begin{thm} \label{nilpfinite}
The scheme $\Nilp_r(X)$ is finite.
\end{thm}

\begin{proof}
By Lemma \ref{operaffine} the scheme $\Nilp_r(X)$ is a closed subvariety of an affine
space, hence it is affine. It suffices therefore to show that it is
proper over $\Spec(k)$. Using composition with the section
$\sigma_\theta : \Op_{\PGL(r)} \ra \Op_{\SL(r)}$ it will be enough to show that the
fiber $\mathcal{N}\mathrm{ilp}_r(X)$  over $0$
of the morphism $\Op_{\SL(r)} \ra {\bf W}_r$ (see section 3.3) is universally closed
over $\Spec(k)$. Using the projection map $pr: \Op_{\SL(r)} \rightarrow \Op_{\PGL(r)}$
this implies that $\mathrm{Nilp}_r(X)$ is also universally closed.
The assertion on $\mathcal{N}\mathrm{ilp}_r(X)$ will be a consequence of the following
valuative criterion.
\begin{prop} \label{ssreduction}
Let $R$ be a discrete valuation ring and let $s$ and $\eta$ be the closed and
generic point of $\Spec(R)$. For any nilpotent $\SL(r)$-oper
$(V_\eta, \nabla_\eta, (\nabla_\eta)_\mydot)$ over $X \times \Spec(K)$
there exists a nilpotent $\SL(r)$-oper $(V_R, \nabla_R, (\nabla_R)_\mydot)$ over
$X \times \Spec(R)$ extending $(V_\eta, \nabla_\eta, (\nabla_\eta)_\mydot)$.
\end{prop}

\begin{proof}
First of all we observe that the local system $(V_\eta, \nabla_\eta)$ is stable
by Proposition \ref{stabilityoper}. Hence we can apply \cite{laszlo-pauly}
Proposition 5.2 which asserts the existence of a local system $(V_R, \nabla_R)$
over $X \times \Spec(R)$ extending $(V_\eta, \nabla_\eta)$ and such that $(V_s, \nabla_s)$ is semistable and
$\psi(V_s, \nabla_s)$ is nilpotent, where $(V_s, \nabla_s)$ denotes the
restriction of $(V_R, \nabla_R)$ to the special fiber.


By Theorem \ref{thmshatz} (2) the Harder-Narasimhan polygon raises under
specialization, i.e. $\sP_{V_s} \succcurlyeq \sP_{V_\eta} = \sP_r^{oper}$. On the
other hand, by Theorem~\ref{thmoperdominant} (1) we have $\sP_r^{oper}
\succcurlyeq \sP_{V_s}$ since the local system $(V_s, \nabla_s)$ is
semistable. Hence we obtain equality $\sP_{V_s} = \sP_{V_\eta} = \sP_r^{oper}$.


It remains to check that the extension $(V_R)_\mydot$ to $X \times \Spec(R)$
of the oper flag $(V_\eta)_\mydot$ has the property that the restriction to the
special fiber $(V_s, \nabla_s, (V_s)_\mydot)$ is an oper. We note that by
properness of the Quot-scheme the extension $(V_R)_\mydot$ exists and is unique.
Restricting $(V_R)_\mydot$ to the special fiber gives a filtration by
{\em subsheaves}
\begin{equation} \label{filtVs}
0 =  (V_s)_r  \subset (V_s)_{r-1} \subset \cdots \subset (V_s)_1 \subset (V_s)_0 = V_s
\end{equation}
Since the degrees are constant under specialization $\deg (V_s)_i = \deg (V_\eta)_i$
for $0 \leq i \leq r-1$ and since $\sP_{V_s} = \sP_{V_\eta}$, we deduce that the
subsheaves of the filtration \eqref{filtVs} are subbundles, i.e. the quotients
$(V_s)_i/(V_s)_{i+1}$ are torsion free. Hence the filtration \eqref{filtVs}
coincides with the Harder-Narasimhan filtration of $V_s$. Finally, Theorem~\ref{thmoperdominant} (2)
allows to conclude that $(V_s, \nabla_s, (V_s)_\mydot)$ is an oper.
\end{proof}
This completes the proof of the theorem.
\end{proof}

Since dormant opers are nilpotent, we immediately obtain

\begin{cor} \label{finiteFroboper}
There exists only a finite number of dormant opers with fixed trivial determinant.
\end{cor}

\begin{rem} \label{HMproper}
The previous proof generalizes straightforwardly to arbitrary fibers (see
\cite{F} Theorem I.3 or \cite{laszlo-pauly} Proposition 5.2), which implies that
the Hitchin-Mochizuki morphism $\mathrm{HM}$ is proper.
\end{rem}

\subsection{Dimension of Quot-schemes}

We refer to section \ref{sectionFroboperQuot} for the correspondence between dormant opers and Quot-schemes.

\begin{thm} \label{Quotexpdim}
Assume $p > \crg$. For any line bundle $Q$ with $\deg(Q) = -(r-1)(g-1)$ the Quot-scheme
$\mathrm{Quot}^{r,0}(F_*(Q))$ is $0$-dimensional.
\end{thm}

\begin{proof}
We deduce from Proposition \ref{bijFrobopersQuot} and Corollary \ref{finiteFroboper} that
$\dim \QQuot(r,\cO_X) = 0$. Hence $\dim \QQuot(r,0) = g$ and, since
the isomorphism class of the fiber $\alpha^{-1}(Q)$ does not depend
on $Q$, we obtain $\dim \mathrm{Quot}^{r,0}(F_*(Q)) = 0$.
\end{proof}

\begin{rem}
Let us remark that
in order to show that $\dim \mathrm{Quot}^{r,0}(F_*(Q)) = 0$, it is sufficient to
show that the natural morphism  $\mathrm{Quot}^{r,0}(F_*(Q)) \ra  \Op_{\PGL(r)}$ is injective
and use properness of the Quot-scheme and the fact that $\Op_{\PGL(r)}$ is affine to conclude that $\dim \mathrm{Quot}^{r,0}(F_*(Q)) = 0$. 
\end{rem}


\section{Applications to loci of Frobenius-destabilized rank-$2$ vector
bundles}\label{dimensions-rank-two}

\subsection{Dimension of any irreducible component}
In this section we will deal with rank-$2$ vector bundles and use the notation introduced in
section \ref{loci}. Note that in this case
$C(2,g) = 0$ and that $\cJ^s(2) = \cJ(2)$ since there are
no strictly semistable rank-$2$ Frobenius-destabilized vector bundles.

\begin{rem}
It is shown in \cite{Mo2} that $\dim \cJ(2) = 3g-4$ for a general
curve $X$ under the assumption $p> 2g-2$.
\end{rem}

As an application of our results on opers we obtain the following information
on the locus of Frobenius-destabilized bundles $\cJ(2)$.

\begin{thm} \label{dimJ2}
Any irreducible component of $\cJ(2)$ containing a dormant oper
has dimension $3g-4$.
\end{thm}

\begin{proof}
We recall from Theorem \ref{locusJ} that there is a surjective morphism
$$ \pi : \QQuot(1,2,0) \lra \cJ(2),$$
and that $\alpha : \QQuot(1,2,0) \rightarrow \Pic^{-1}(X)$ is a fibration, whose fibers
$\alpha^{-1}(Q) = \mathrm{Quot}^{2,0}(F_*(Q))$ are all isomorphic.


We will need a more general version of Theorem \ref{Quotexpdim}

\begin{lem} \label{dimQuotgen}
Let $Q$ be a line bundle of degree $\deg(Q) = -(g-1) + d$ with $0 \leq d \leq 2g-3$. Then any
irreducible component of $\mathrm{Quot}^{2,0}(F_*(Q))$ containing a dormant oper has dimension $2d$.
\end{lem}

\begin{proof}
We prove the result by induction on $d$. For $d=0$, this is exactly Theorem \ref{Quotexpdim}.
Consider a line bundle $Q$ with $\deg(Q) = -(g-1) + (d+1)$  and let $\cC \subset
\mathrm{Quot}^{2,0}(F_*(Q))$ be an irreducible component containing a dormant oper
$E$. Since $F^*(E)$ is the underlying bundle of an oper, we have an
exact sequence
$$ 0 \ra M \ra F^*(E) \ra L \ra 0 $$
with $\deg (M) = g-1 = - \deg (L)$. By the assumption on the degree of $Q$, we have $\deg (Q) < \deg (M)$ and therefore
$\Hom(M,Q) = \{ 0 \}$. So the morphism $F^*(E)\to Q$ factors as $F^*(E)\onto L\into Q$. So we obtain a 
non-zero section $\cO_X\to Q\tensor L^{-1}$, and since $\deg( Q \otimes L^{-1} ) = d+1$, we can write $L = Q(-D)$ for
some an effective divisor $D$ of degree $d+1$. So, by adjunction, we have $E\into F_*(Q(-D))\subset F_*(Q)$.  We decompose
$D = x + D'$ with $x \in X$ and $D'$ effective of degree $d$. Let $\cC'$ be an
irreducible component of $\mathrm{Quot}^{2,0}(F_*(Q(-x))) \cap \cC$ containing $E$.
By induction we have $\dim \cC' = 2d$.

Now we claim that $\mathrm{codim}_{\cC}(\cC') \leq 2$. To prove this note that
$\cC' \not= \emptyset$. Since $\cC$ is an irreducible component of the Quot-scheme, it is equipped with a universal quotient
sheaf $\cQ$ over $X \times \cC$
$$ 0 \lra \cE \lra p_X^*(F_*(Q)) \lra \cQ \lra 0.$$
We denote by $\cE$ the kernel $\ker (p_X^*(F_*(Q)) \onto \cQ)$. Since $\cQ$ and $p_X^*(F_*(Q))$ are 
$\cC$-flat, $\cE$ is also a $\cC$-flat and $\forall c \in \cC$ the homomorphism
$\cE_{| X \times \{ c \}} \ra F_*(Q)$ is injective. Hence, since $F_*(Q)$ is locally free, $\cE_{| X \times \{ c \}}$
is also locally free (since torsion free over a smooth curve) and by \cite{huy-book} Lemma 2.1.7 we conclude that
$\cE$ is locally free over $X \times \cC$. Since $p_X^*(F_*(Q)) = (F \times \mathrm{id}_{\cC})_*(p_X^*(Q))$ we obtain 
by adjunctioon a non-zero map $(F \times \mathrm{id}_{\cC})^* \cE \ra  p_X^*(Q)$, hence a section $\sigma$ of the
rank-$2$ vector bundle $\cV := Hom ((F \times \mathrm{id}_{\cC})^* \cE ,p_X^*(Q))$ over $X \times \cC$. It is clear that 
$\mathrm{Quot}^{2,0}(F_*(Q(-x))) \cap \cC$ is the zero-scheme of the restricted section $\sigma_{| \{ x \} \times \cC } \in
H^0(\cC, \cV_{| \{ x \} \times \cC } )$. Hence $\mathrm{codim}_{\cC}(\cC') \leq 2$ and therefore 
$\dim \cC \leq 2d+2$. On the other hand, by the
dimension estimate of the Quot-schemes given in Proposition \ref{dimestimateQuot} we
have $\dim \cC \geq 2d+2$. Therefore $\dim \cC = 2d+2$ and we are done.
\end{proof}

Using the fibration $\alpha$ we deduce from Lemma \ref{dimQuotgen} applied with $d = g-2$ that any
irreducible component $\cI$ of $\QQuot(1,2,0)$ containing a dormant oper has dimension
$3g-4$. Therefore it will suffice to show that the restriction of $\pi$ to
$\cI$ is generically injective. This, in turn, will follow from the fact that a general
vector bundle $E \in \cI$ satisfies
$$ f_E: F^*(E) \rightarrow Q \qquad \text{surjective},$$
where the map $f_E$ is obtained by adjunction from the sheaf inclusion $E \hookrightarrow F_*(Q)$
with $Q \in \Pic^{-1}(X)$. In fact, if $f_E$ is surjective, then $0 \subset \ker f_E \subset F^*(E)$
is the Harder-Narasimhan filtration of $F^*(E)$, which implies that the quotient $Q$ is unique and
that $\dim \Hom(E, F_*(Q)) = 1$.


Let us now show that the map $f_E$ is surjective for a general $E \in \cC \subset
\mathrm{Quot}^{2,0}(F_*(Q))$, where $\cC$ is an irreducible component containing a dormant oper
and $\deg(Q) = -1$. Suppose on the contrary that this is not the case. Then any $E \in \cC$
lies in $\mathrm{Quot}^{2,0}(F_*(Q(-x)))$ for some $x \in X$, i.e.
\begin{equation} \label{coveringCC}
 \cC = \bigcup_{x \in X} \mathrm{Quot}^{2,0}(F_*(Q(-x))) \cap \cC.
\end{equation}
Two cases can occur:

$(1)$ there exists a point $x \in X$ such that $\dim \mathrm{Quot}^{2,0}(F_*(Q(-x))) \cap \cC = 2g-4$. Then
$\cC = \mathrm{Quot}^{2,0}(F_*(Q(-x))) \cap \cC$ and contains a dormant oper. This contradicts Lemma
\ref{dimQuotgen} with $d = g-3$.

$(2)$ for any point $x \in X$ we have  $\dim \mathrm{Quot}^{2,0}(F_*(Q(-x))) \cap \cC \geq 2g-5$. Because of
\eqref{coveringCC} there exists a point $x \in X$ such that $\mathrm{Quot}^{2,0}(F_*(Q(-x))) \cap \cC$
contains a dormant oper, contradicting again Lemma \ref{dimQuotgen} with $d= g-3$.
\end{proof}

\begin{rem}
Since the set of dormant opers is non-empty, there always exists at least one irreducible
component of $\cJ(2)$ of dimension $3g-4$.
\end{rem}

\begin{ques}
Does any irreducible component of $\cJ(2)$ contain a dormant oper? The only case where the answer is known
is $p=2$: in that case by \cite{JRXY} the locus $\cJ(2)$ is irreducible.
\end{ques}


\bibliographystyle{amsalpha}

\end{document}